\documentclass[reqno,12pt,oneside]{amsart}

\usepackage{amsthm,amsmath,amssymb,amsfonts}
\usepackage[colorlinks=true,hypertexnames=false,linkcolor=blue,citecolor=blue]{hyperref}
\usepackage[top=3cm,bottom=2.5cm,left=2.5cm, right=3cm,marginparwidth=0cm]{geometry}

\usepackage[T1]{fontenc}

\usepackage{epsfig,graphics}
\usepackage{amscd}
\usepackage{amsmath}
\usepackage{mathtools}
\usepackage{bm}
\usepackage{amssymb}
\usepackage{graphicx}
\usepackage{psfrag}

\usepackage{pb-diagram}

\usepackage{verbatim}
\usepackage{calrsfs}
\usepackage{tikz}
\usetikzlibrary{quotes}
\usetikzlibrary{topaths}
\usetikzlibrary{shapes,arrows,cd}

\usepackage{adjustbox}

\newtheorem{theorem}{Theorem}[section]
\newtheorem{lemma}[theorem]{Lemma}
\newtheorem{prop}[theorem]{Proposition}
\newtheorem{coro}[theorem]{Corollary}

\theoremstyle{definition}
\newtheorem{definition}[theorem]{Definition}
\newtheorem{example}[theorem]{Example}
\newtheorem{question}[theorem]{Question}

\theoremstyle{remark}
\newtheorem{remark}[theorem]{Remark}

\theoremstyle{plain}

\numberwithin{equation}{section}

\theoremstyle{plain}

\newcommand{\Q}{\ensuremath{\mathbb{Q}}}
\newcommand{\R}{\ensuremath{\mathbb{R}}}
\newcommand{\Z}{\ensuremath{\mathbb{Z}}}
\newcommand{\C}{\ensuremath{\mathbb{C}}}
\newcommand{\nt}{\ensuremath{\mathbb{N}}}

\newcommand{\crd}{\operatorname{CrD}}

\DeclareMathOperator{\Diff}{Diff}

\DeclareMathOperator{\home}{Hom}

\DeclareMathOperator{\modulo}{mod}
\DeclareMathOperator{\Id}{Id}
\DeclareMathOperator{\varia}{var}

\begin{document}

\title[]{Dynamics of multicritical circle maps}

\author{Edson de Faria}
\address{Instituto de Matem\'atica e Estat\'istica, Universidade de S\~ao Paulo}
\curraddr{Rua do Mat\~ao 1010, 05508-090, S\~ao Paulo SP, Brasil}
\email{edson@ime.usp.br}

\author{Pablo Guarino}
\address{Instituto de Matem\'atica e Estat\'istica, Universidade Federal Fluminense}
\curraddr{Rua Prof. Marcos Waldemar de Freitas Reis, S/N, 24.210-201, Bloco H, Campus do Gragoat\'a, Niter\'oi, Rio de Janeiro RJ, Brasil}
\email{pablo\_\,guarino@id.uff.br}

\thanks{The first author has been supported by ``Projeto Tem\'atico Din\^amica em Baixas Dimens\~oes'' FAPESP Grant  2016/25053-8, while the second author has been supported by Coordena\c{c}\~ao de Aperfei\c{c}oamento de Pessoal de N\'ivel Superior - Brasil (CAPES) grant 23038.009189/2013-05.}

\begin{abstract} This paper presents a survey of recent and not so recent results concerning the study of smooth homeomorphisms of the circle with a finite number of non-flat critical points, an important topic in the area of One-dimensional Dynamics. We focus on the analysis of the fine geometric structure of orbits of such dynamical systems, as well as on certain ergodic-theoretic and complex-analytic aspects of the subject. Finally, we review some conjectures and open questions in this field.
\end{abstract}

\maketitle

\vspace{-0.5cm}

\section{Introduction}

One of the major general goals of the area of Dynamical Systems is to solve the {\it smooth classification problem\/}: given two smooth dynamical systems which are topologically equivalent, when are they smoothly equivalent? In somewhat vague terms, this problem is tantamount to understanding the fine-scale geometric properties of such systems. 

In such general setting, and particularly in higher dimensions, the above classification problem seems rather daunting (perhaps even hopeless). Hence one should first attempt to understand low-dimensional systems. At least at an intuitive level, the problem should be much simpler for one-dimensional systems; after all, in dimension one the linear order structure and ``lack of ambient space'' should impose severe restrictions on the possible geometries of such systems, thereby facilitating their smooth classification. However, even here the problem turns out to be rather subtle. A basic distinction that must be made in the one-dimensional context is between {\it invertible\/} dynamics -- to wit, homeomorphisms of the circle -- and {\it non-invertible\/} dynamics, such as the dynamics of unimodal or multimodal maps of the interval (or the circle). 

In the present survey we deal with invertible dynamical systems on the circle. In the case of smooth {\it diffeomorphisms\/} of the circle, deep results have been obtained from the mid to late seventies onwards,  starting with M.~Herman's thesis \cite{hermanihes} and culminating with the work of J.-C.~Yoccoz \cite{yoccoz2}, with important contributions by Y.~Katznelson and D.~Ornstein \cite{ko1}, among others. It was in part due to his deep work on diffeomorphisms that Yoccoz was awarded his Fields Medal in 1994. 

We will not say much about the smooth classification of diffeomorphisms of the circle in the present paper (but see Section \ref{secAHY}). We will rather focus on the case of smooth homeomorphisms with {\it critical points\/}, a topic to which both authors have made contributions that are deemed significant by the one-dimensional dynamics community. In this context, the notions of {\it renormalization\/}, {\it rigidity\/} and {\it universality\/} play a decisive role, and have been widely studied in the last thirty years. From a strictly mathematical viewpoint, these studies started with basic topological aspects \cite{hall, yoccoz1}, then evolved to geometric bounds \cite{H, swiatek} and further to geometric rigidity and renormalization aspects \cite{avila, EdsonThesis, edsonETDS, edsonwelington1, edsonwelington2, GY, PabloThesis, GMdM2015, GdM2013, khaninteplinsky, khmelevyampolsky, swiatek, yampolsky1, yampolsky2, yampolsky3, yampolsky4}.

Let us summarize some of those contributions in the following statements: on the one hand, any two $C^3$ circle homeomorphisms with the same irrational rotation number of \emph{bounded type} and with a single critical point (of the same odd power-law type) are conjugate to each other by a $C^{1+\alpha}$ circle diffeomorphism, for some universal $\alpha>0$ \cite{GdM2013}. On the other hand, any two $C^4$ circle homeomorphisms with the same irrational rotation number and with a single critical point (again, of the same odd type), are conjugate to each other by a $C^1$ diffeomorphism \cite{GMdM2015}. This conjugacy is a $C^{1+\alpha}$ diffeomorphism for a certain set of rotation numbers that has full Lebesgue measure, but {\it does not\/} include all irrational rotation numbers (see Section \ref{sec:conc} for its definition and further remarks). These statements can be regarded as the state of the art concerning renormalization and rigidity of critical circle maps with \emph{a single} critical point, and they will be further discussed in Section \ref{smoothrigid} of the present survey. In the case of homeomorphisms with \emph{several} critical points, the theory remains wide open. However, some interesting recent results \cite{EdF, EdFG, EG, ESY, dFG2019, yampolsky5} will be described in this survey (again, see Section \ref{smoothrigid}).

We close this introduction by remarking that the rigidity and renormalization theories to be (partially) described in this survey have been the subject of at least six ICM talks, namely \cite{avilaICM, HermanICM, khaninICM, mcicm, dmicm,sullivanICM}, and it was in part due to their seminal contributions to these theories that C.~McMullen (1998) and A.~Avila (2014) were awarded their Fields Medals. 

\section{Circle Diffeomorphisms}\label{secdifeos}

The \emph{circle} is defined as $S^1 = \R / \Z$\,, and it can be identified with $\partial\mathbb{D}=\big\{z \in\C: |z|=1\big\}$ via the universal covering map $\pi:\R\to\partial\mathbb{D}$ given by $\pi (x) = e^{2 \pi ix}$. We denote by $\home_+(\R)$ the space of orientation preserving homeomorphisms of the real line, and by $\home_+(S^1)$ the space of orientation preserving circle homeomorphisms (both endowed with the $C^0$ topology). For any given $f\in\home_+(S^1)$ there exist infinitely many $F\in\home_+(\R)$ satisfying $f\circ\pi=\pi \circ F$. Such an $F$ is called a \emph{lift} of $f$, and any two lifts differ by an integer translation. In fact, $F\in\home_+(\R)$ is the lift of a circle homeomorphism if, and only if, it commutes with unitary translation: $F(x+1)=F(x)+1$ for all $x\in\R$.

\subsection{The rotation number} The first major result in circle dynamics goes back to the end of the nineteenth century, and is due to Poincar\'e.

\begin{theorem}[Poincar\'e]\label{TeoPoinnumrot} Let $F\in\home_+(\R)$ be a lift of some $f\in\home_+(S^1)$. Then for all $x\in\R$ the following limit exists and it is independent of $x$:$$\tau(F)=\lim_{n \to+\infty}\frac{F^n(x)-x}{n}\,.$$
\end{theorem}

Since any two lifts $F_1$ and $F_2$ of $f$ differ by an integer translation, it is easy to see that $\tau(F_1)$ and $\tau(F_2)$ differ by the same integer. The \emph{rotation number} of $f$ is then defined as $\rho(f)=\tau(F)\,\,\textrm{(mod 1)}$, for any given lift $F$ of $f$. The function $\rho:\home_+(S^1)\to[0,1)$ is continuous in the $C^0$ topology. The rotation number is preserved under conjugation with orientation preserving homeomorphisms, and therefore it is a \emph{topological invariant} (in the sense that the equivalence classes of $\home_+(S^1)$ under topological conjugacies are contained in the level sets of $\rho$).

\begin{remark} Let us briefly explain a measure-theoretical way to define the rotation number of a circle homeomorphism.  Let $f\in\home_{+}(S^1)$ and let $F\in\home_{+}(\R)$ be one of its lifts. Let $\varphi_F:S^1\to\R$ be the real function whose lift under $\pi$ is $F-\Id$, that is: $F(x)=x+\varphi_F(e^{2\pi ix})$ for all $x\in\R$\,. We claim that for any given $f$-invariant Borel probability measure $\mu$ we have
\begin{equation}\label{defnumrotmedida}
\rho(f)=\int_{S^1}\varphi_{F}\,\,d\mu\,\,\,(\modulo 1)\,.
\end{equation}
Indeed, the point here is that the real function $\widetilde{\varphi_F}$ given by $$\widetilde{\varphi_F}=\lim_{n\to+\infty}\frac{1}{n}\sum_{j=0}^{n-1}\varphi_F \circ f^j$$ is well defined on the whole circle and it is \emph{constant}, equal to $\tau(F)$ (note that this implies at once the identity \eqref{defnumrotmedida}, since $\int_{S^1}\varphi_{F}\,\,d\mu=\int_{S^1}\widetilde{\varphi_F}\,\,d\mu$ by the $f$-invariance of $\mu$). To see that $\widetilde{\varphi_F}$ is well defined and constant, let $x \in S^1$ and let $x_0\in\R$ be such that $\pi(x_0)=x$. Then
\begin{align*}
\sum_{j=0}^{n-1}\varphi_F\big(f^j(x)\big)&=\sum_{j=0}^{n-1}\varphi_F\big(\pi(F^j(x_0))\big)=\sum_{j=0}^{n-1}\big(F-\Id\big)\big(F^j(x_0)\big)\\
&=\sum_{j=0}^{n-1}\big(F^{j+1}(x_0)-F^j(x_0)\big)=F^n(x_0)-x_0\,.
\end{align*}
By Theorem \ref{TeoPoinnumrot} we have $\widetilde{\varphi_F}(x)=\tau(F)$, as desired. The equivalent definition of the rotation number of a circle homeomorphism given by \eqref{defnumrotmedida} will not be further mentioned in this survey, but it is important in its own right.
\end{remark}

\subsection{Circle diffeomorphisms with irrational rotation number}\label{sec:Denjoy} The rotation number of $f\in\home_{+}(S^1)$ is rational if, and only if, $f$ admits at least one periodic orbit. In this case, if $\rho(f)=p/q$, then any periodic orbit has period equal to $q$ and the dynamics of $f$ is somehow trivial: any given initial condition is asymptotic (both for the past and the future) to a periodic orbit.

If $f$ has irrational rotation number $\rho$, its orbit structure is much more interesting. The fact that $f$ has no periodic orbits implies that its non-wandering set is \emph{minimal}, being a Cantor set or the whole circle. Now pick a non-wandering point $x$ in $S^1$ and consider its orbit under $f$, that is: $\mathcal{O}_f(x)=\big\{f^n(x)\big\}_{n\in\Z}$. The function $h_x:\mathcal{O}_f(x) \to S^1$ given by $h_x\big(f^n(x)\big)=\pi(n\rho)=e^{2\pi in\rho}$ identifies $x$ with the point $1$, and conjugates $f$ with the rigid rotation $R_{\rho}$ along the orbit of $x$. A crucial fact here is that $f$ and $R_{\rho}$ are \emph{combinatorially equivalent}, in the sense that for each $n\in\nt$ the first $n$ elements of the orbit of $x$ under $f$ are ordered in the same way as the first $n$ elements of the orbit of the point $1$ under the rotation $R_{\rho}$ (otherwise $f$ must have a periodic orbit, which is impossible in the irrational case). The combinatorial equivalence between $f$ and $R_{\rho}$ implies that the map $h_x$, being monotone, extends continuously to the closure of $\mathcal{O}_f(x)$. This extension is surjective because any orbit of $R_{\rho}$ is dense in $S^1$, and then we can extend it (again) as a constant function in any connected component of the complement of the closure of $\mathcal{O}_f(x)$. This produces a \emph{semi-conjugacy} between $f$ and $R_{\rho}$, that is: a continuous surjective map $h_x:S^1 \to S^1$ satisfying $R_{\rho} \circ h_x=h_x \circ f$.
$$
\begin{CD}
S^1@>{f}>>S^1\\
@V{h_x}VV             @VV{h_x}V\\
{S^1}@>>{R_{\rho}}>{S^1}
\end{CD}
$$
From the previous construction it is clear that $h_x$ is the \emph{unique} semi-conjugacy between $f$ and $R_{\rho}$ which identifies the point $x$ with the point $1$ (this uniqueness will be useful during the proof of Proposition \ref{lemaunerg} below). Note also that for any given point $z \in S^1$ we have $h_z = R_{\beta} \circ h_x$, with $\exp(2\pi i\beta)=1/h_x(z)$. Therefore, any semi-conjugacy between $f$ and $R_{\rho}$ is unique up to post-composition with rotations.

If $f$ is minimal in the whole circle we have $\overline{\mathcal{O}_f(x)}=S^1$. In particular, $h_x$ is a homeomorphism (and then $f$ is topologically conjugate to the rotation $R_{\rho}$). Otherwise, there exists a point $y \in S^1$ such that $I=h_x^{-1}\big(\{y\}\big)$ is a non-degenerate closed interval. We call $I$ a \emph{wandering interval} since $f^n(I) \cap f^m(I) = \emptyset$ if $n \neq m \in \Z$, and since it is not contained in the basin of a periodic attractor. The second major result in circle dynamics was obtained by Denjoy in the early thirties \cite{denjoy}.

\begin{theorem}[Denjoy]\label{denjoy} Let $f$ be a $C^1$ circle diffeomorphism with irrational rotation number $\rho$. If\, $\log Df$ has bounded variation, then $f$ has no wandering intervals and therefore it is topologically conjugate to the rigid rotation of angle $\rho$.
\end{theorem}

\begin{remark}\label{remarkdenjoy} In \cite{denjoy}, Denjoy also proved that \emph{some} assumption on the first derivative of $f$ is necessary: given any irrational number $\rho$ there exists a $C^1$ circle diffeomorphism with rotation number $\rho$ and wandering intervals \cite[Section I.2]{livrounidim} (actually, examples of this type were given before by Bohl in \cite{Bohl}). Moreover, counterexamples can be constructed in such a way that the first derivative is H\"older continuous \cite{hermanihes}. See also \cite{husul,ko1,ko2}.
\end{remark}

\begin{remark} As explained above, any circle homeomorphism with irrational rotation number has minimal dynamics when restricted to its non-wandering set, which may be a Cantor set or the whole circle. For $r \geq 0$, let
$$\mathcal{C}_r=\big\{K \subset S^1:\mbox{$K$ is the minimal Cantor set for some $C^r$ circle diffeomorphism}\big\}.$$
Note the obvious fact that $\mathcal{C}_r\supseteq \mathcal{C}_s$ whenever $r<s$.
Every Cantor set of the unit circle is contained in $\mathcal{C}_0$, since any two Cantor sets in $S^1$ are homeomorphic under some circle homeomorphism. As already mentioned in Remark \ref{remarkdenjoy}, $\mathcal{C}_1$ is non-empty, while Theorem \ref{denjoy} implies that $\mathcal{C}_r$ is empty for any $r \geq 2$. The problem of describing $\mathcal{C}_1$ seems to go back to Michael Herman, and it was first addressed by Dusa McDuff in the late seventies. For instance, she showed in \cite{duff} that the standard middle-thirds Cantor set (more precisely, its isometric copy in $S^1=\mathbb{R}/\mathbb{Z}$) is not in $\mathcal{C}_1$. To the best of our knowledge, a full description of $\mathcal{C}_1$ remains open (see \cite{BIP} and references therein for partial results and recent developments).
\end{remark}

\medskip

The standard proof of Theorem \ref{denjoy} relies on the following simple observation: any given $f\in\home_+(S^1)$ such that some subsequence $\{f^{n_k}\}_{k\in\nt}$ is \emph{equicontinuous}, has no wandering intervals. Indeed, if some interval $I \subset S^1$ wanders under $f$, then the length of $f^{-n}(I)$ goes to zero as $n$ goes to infinity, since the images of $I$ (past and futire) are pairwise disjoint; but this contradicts the fact that some sequence of iterates of $f$ is equicontinuous. Therefore, in order to prove Theorem \ref{denjoy} we need first to choose a suitable subsequence of iterates of $f$. For this purpose, we recall that, being irrational, the rotation number $\rho\in(0,1)$ has an infinite \emph{continued fraction expansion}, say
\begin{equation*}
\rho= [a_{0} , a_{1} , \cdots ]=
\cfrac{1}{a_{0}+\cfrac{1}{a_{1}+\cfrac{1}{ \ddots} }} \ .
\end{equation*}

A classical reference for continued fraction expansion is the monograph \cite{khin}. Truncating the expansion at level $n-1$,
we obtain a sequence of fractions $p_n/q_n$, which are called the \emph{convergents} of the irrational $\rho$\,:
$$
\frac{p_n}{q_n}\;=\;[a_0,a_1, \cdots ,a_{n-1}]\;=\;\dfrac{1}{a_0+\dfrac{1}{a_1+\dfrac{1}{\ddots\dfrac{1}{a_{n-1}}}}}\,.
$$
The sequence of denominators $q_n$, which we call the \emph{return times} of $\rho$, satisfies
\begin{equation*}
 q_{0}=1, \hspace{0.4cm} q_{1}=a_{0}, \hspace{0.4cm} q_{n+1}=a_{n}\,q_{n}+q_{n-1} \hspace{0.3cm} \text{for $n \geq 1$} .
\end{equation*}

As explained above, Theorem \ref{denjoy} follows from the following result.

\begin{prop}\label{1+bv} Let $f$ be a $C^1$ circle diffeomorphism with irrational rotation number. Assume that\, $\log Df$ has bounded variation $V=\varia(\log Df)>0$, and let $K=e^V>1$. Then we have$$\frac{1}{K} \leq Df^{q_n}(x) \leq K\quad\mbox{for all $x \in S^1$ and all $n\in\nt$,}$$where $\{q_n\}_{n\in\nt}$ is the sequence of return times given by the rotation number of $f$.
\end{prop}

Proposition \ref{1+bv} will be a straightforward consequence of two well-known results (Proposition \ref{DK} and Lemma \ref{zeroexpdifeo} below) which are interesting in its own right. Before that, we establish the \emph{unique ergodicity} of circle homeomorphisms without periodic orbits.

\begin{prop}[Unique ergodicity]\label{lemaunerg} If $f\in\home_+(S^1)$ has irrational rotation number $\rho$, it preserves a unique Borel probability measure.
\end{prop}

\begin{proof}[Proof of Proposition \ref{lemaunerg}] We claim first that\, $\nu\big(x,f(x)\big)=\rho$\, for any $x \in S^1$ and any $f$-invariant Borel probability measure $\nu$. Indeed, let $F\in\home_+(\R)$ be a lift of $f$ and, for any given $x \in S^1$, let $\tilde{x}\in\R$ be one of its lifts. We will prove that $\tau(F)=\tilde{\nu}\big(\tilde{x},F(\tilde{x})\big)$, where $\tilde{\nu}$ is the lift of $\nu$ to the real line (an infinite but $\sigma$-finite Borel measure). For each $n\in\nt$, let us consider $j_n=\lfloor F^n(\tilde{x})-\tilde{x} \rfloor$, so that $F^n(\tilde{x})\in(\tilde{x}+j_n,\tilde{x}+j_n+1)$. Then we have:
\[
F^n(\tilde{x})-\tilde{x}-1<j_n=\tilde{\nu}(\tilde{x},\tilde{x}+j_n)<\tilde{\nu}\big(\tilde{x},F^n(\tilde{x})\big)<\tilde{\nu}(\tilde{x},\tilde{x}+j_n+1)=j_n+1<F^n(\tilde{x})-\tilde{x}+1\,,
\]
and therefore 
\[
\lim_{n \to+\infty}\frac{F^n(\tilde{x})-\tilde{x}}{n}=\lim_{n \to+\infty}\frac{\tilde{\nu}\big(\tilde{x},F^n(\tilde{x})\big)}{n}\,.
\]
On the other hand, since $F$ is monotone and $\tilde{\nu}$ is $F$-invariant, we see that 
$$\tilde{\nu}\big(\tilde{x},F^n(\tilde{x})\big)=\sum_{j=0}^{n-1}\tilde{\nu}\big(F^{j}(\tilde{x}),F^{j+1}(\tilde{x})\big)=n\,\tilde{\nu}\big(\tilde{x},F(\tilde{x})\big),$$and this tells us that 
$$\lim_{n \to+\infty}\frac{F^n(\tilde{x})-\tilde{x}}{n}=\lim_{n \to+\infty}\frac{\tilde{\nu}\big(\tilde{x},F^n(\tilde{x})\big)}{n}=\tilde{\nu}\big(\tilde{x},F(\tilde{x})\big)\,.$$
Since $\nu\big(x,f(x)\big)=\tilde{\nu}\big(\tilde{x},F(\tilde{x})\big)\,\,\textrm{(mod 1)}$, the claim is proved. Now fix some $x \in S^1$ and let $h_x:S^1 \to S^1$ be the semi-conjugacy between $f$ and $R_{\rho}$ constructed above (before the statement of Theorem \ref{denjoy}). We define a Borel probability measure $\mu$ in $S^1$ as the \emph{pull-back} of the Lebesgue measure under $h_x$, that is: $\mu(A)=\lambda\big(h_x(A)\big)$ for any Borel set $A \subset S^1$, where $\lambda$ denotes the normalized Lebesgue measure in the unit circle. Since $R_{\rho}$ preserves $\lambda$ and since $h_x$ is a semi-conjugacy between $f$ and $R_{\rho}$, we deduce that $f$ preserves $\mu$. We claim that $\mu$ is the \emph{unique} Borel probability measure which is invariant under $f$. Indeed, let $\nu$ be any $f$-invariant probability measure, and consider $\tilde{h}_x:S^1 \to S^1$ defined by$$\tilde{h}_x(y)=\exp\left(2\pi i\int_{x}^{y}\!d\nu\right),$$where, by convention, we measure the arc $(x,y)$ starting from $x$ in the counter-clockwise sense. Since $\nu$ is $f$-invariant and $f$ has no periodic orbits, $\nu$ has no points of positive measure, which implies that $\tilde{h}_x$ is continuous, surjective and that $\tilde{h}_x(x)=1$. Moreover, using that\, $\nu\big(x,f(x)\big)=\rho$\, for any $x \in S^1$, we deduce that $\tilde{h}_x \circ f = R_{\rho} \circ \tilde{h}_x$, and then it follows that $\tilde{h}_x=h_x$. Finally, note that by definition $\nu$ is the pull-back of the Lebesgue measure under $\tilde{h}_x$, that is: $\lambda\big([\tilde{h}_x(x),\tilde{h}_x(y)]\big)=\lambda\big([1,\tilde{h}_x(y)]\big)=\nu\big([x,y]\big)$, which implies that it coincides with the measure $\mu$ defined above.
\end{proof}

\begin{remark} It is not difficult to prove that every circle homeomorphism has zero topological entropy.
\end{remark}

\begin{prop}[Denjoy-Koksma inequality]\label{DK} Let $f\in\home_{+}(S^1)$ with irrational rotation number $\rho$, and let $\mu$ be its unique invariant Borel probability measure. Let $\{q_n\}_{n\in\nt}$ be the sequence of return times given by $\rho$. For any $\psi:S^1\to\R$ (non necessarily continuous) with finite total variation\, $\varia(\psi)$ we have:$$\left|\sum_{j=0}^{q_n-1}\psi\big(f^j(x)\big)-q_n\int_{S^1}\psi\,d\mu\right|\leq\varia(\psi)\quad\mbox{for all $x \in S^1$ and all $n\in\nt$.}$$
\end{prop}

\begin{proof}[Proof of Proposition \ref{DK}] Fix $x \in S^1$ and $n\in\nt$. By combinatorics there exist $q_n$ pairwise disjoint open intervals $\{I_0,I_1,...,I_{q_n-1}\}$ in the unit circle such that $R_{\rho}^{j}(1)=e^{2\pi ij\rho}\in\overline{I_j}$ and $\lambda(I_j)=1/q_n$ for all $j\in\{0,1,...,q_n-1\}$ (just take the intervals determined by the $q_n$-roots of unity, and label them in order to have $e^{2\pi ij\rho}\in\overline{I_j}$ for all $j\in\{0,1,...,q_n-1\}$). Let $h=h_x$ be the semi-conjugacy between $f$ and $R_{\rho}$ that maps the point $x$ to the point $1$, and for each $j\in\{0,1,...,q_n-1\}$ let $J_j=h^{-1}(I_j)$. Note that $f^j(x) \in J_j$ and $\mu(J_j)=1/q_n$ for all $j\in\{0,1,...,q_n-1\}$. Moreover $\big\{\overline{J_j}\big\}_{j=0}^{q_n-1}$ is a partition of the unit circle (modulo boundary points, whose $\mu$-measure is zero since $\mu$ is $f$-invariant and $f$ has no periodic orbits). Therefore:
\begin{align}
\left|\sum_{j=0}^{q_n-1}\psi\big(f^j(x)\big)-q_n\int_{S^1}\psi\,d\mu\right|&=\left|\sum_{j=0}^{q_n-1}\left(\psi\big(f^j(x)\big)-q_n\int_{J_j}\psi\,d\mu\right)\right|\notag\\
&\leq\sum_{j=0}^{q_n-1}\left|\psi\big(f^j(x)\big)-q_n\int_{J_j}\psi\,d\mu\right|\notag\\
&=q_n\sum_{j=0}^{q_n-1}\left|\int_{J_j}\big(\psi\big(f^j(x)\big)-\psi\big)d\mu\right|\notag\\
&\leq q_n\sum_{j=0}^{q_n-1}\int_{J_j}\big|\psi\big(f^j(x)\big)-\psi\big|d\mu\notag\\
&\leq\sum_{j=0}^{q_n-1}\sup_{y \in J_j}\big|\psi\big(f^j(x)\big)-\psi(y)\big|\leq\varia(\psi)\,.\notag
\end{align}
\end{proof}

\begin{remark} If $f$ is a sufficiently smooth diffeomorphism and the observable $\psi$ is also smooth, then the Denjoy-Koksma inequality can be considerably improved: see \cite[Section 1.2]{AK} and references therein. See also the recent paper \cite{KLM}, where precise estimates on the (uniform) rate of convergence of the Birkhoff averages for H\"older observables under Diophantine rotations are given.
\end{remark}

\begin{lemma}[Zero Lyapunov exponent]\label{zeroexpdifeo} Let $f\in\Diff_{+}^{1}(S^1)$ with irrational rotation number, and let $\mu$ be its unique invariant Borel probability measure. Then:$$\int_{S^1}\log Df\,d\mu=0\,.$$
\end{lemma}

\begin{proof}[Proof of Lemma \ref{zeroexpdifeo}] Since $f\in\Diff_{+}^{1}(S^1)$, the function $\psi:S^1\to\R$ defined by $\psi=\log Df$ is a continuous function and therefore, by the unique ergodicity of $f$ (Proposition \ref{lemaunerg}), the sequence of functions$$\frac{1}{n}\sum_{j=0}^{n-1}\psi\circ f^j$$converges \emph{uniformly} to the constant $\int_{S^1}\log Df\,d\mu$. By the chain rule:
\begin{equation}\label{chain}
\sum_{j=0}^{n-1}\psi\circ f^j=\log(Df^n)\,,
\end{equation}and therefore the sequence of continuous functions $\log(Df^n)/n$ converges to the constant $\int_{S^1}\log Df\,d\mu$ uniformly in $S^1$. Since $f^n$ is a diffeomorphism for all $n\in\nt$, this constant must be zero.
\end{proof}

As we will see in Section \ref{secerg}, Lemma \ref{zeroexpdifeo} still holds true for $C^3$ circle homeomorphisms with critical points (see Theorem \ref{expzeroccm}). With Proposition \ref{DK} and Lemma \ref{zeroexpdifeo} at hand, we can finish Section \ref{sec:Denjoy}. Indeed, to prove Proposition \ref{1+bv} apply the Denjoy-Koksma inequality (Proposition \ref{DK}) with $\psi=\log Df$, and then use \eqref{chain} combined with Lemma \ref{zeroexpdifeo}. As already explained, Proposition \ref{1+bv} immediately implies Denjoy's Theorem \ref{denjoy}. In Section \ref{secyoccozthm} of the present survey we provide another proof of Theorem \ref{denjoy}, this time by means of \emph{cross-ratio} distortion arguments.

\subsection{Some remarks on smooth rigidity for circle diffeomorphisms}\label{secAHY} We finish Section \ref{secdifeos} by mentioning some results concerning \emph{smooth rigidity}, obtained during the second half of the twentieth century by Arnold, Herman and Yoccoz. Let $f$ be a $C^2$ circle diffeomorphism with irrational rotation number $\rho$. By Denjoy's Theorem \ref{denjoy}, there exists a circle homeomorphism $h$ which is a topological conjugacy between $f$ and the rigid rotation of angle $\rho$, that we denote by $R_{\rho}$. The unique invariant Borel probability measure of $R_{\rho}$ (recall Proposition \ref{lemaunerg}) is of course the normalized Lebesgue measure. Therefore, the unique $f$-invariant measure, say $\mu$, is given by the \emph{pull-back} under $h$ of the Lebesgue measure, that is: $\mu(A)=\lambda\big(h(A)\big)$ for any Borel set $A$, where $\lambda$ denotes the Lebesgue measure in the unit circle (recall that the conjugacy $h$ is unique up to post-composition with rotations, so the measure $\mu$ is well-defined).
$$
\begin{CD}
(S^1,\mu)@>{f}>>(S^1,\mu)\\
@V{h}VV             @VV{h}V\\
{(S^1,\lambda)}@>>{R_{\rho}}>{(S^1,\lambda)}
\end{CD}
$$
In particular, $\mu$ has no atoms and gives
positive measure to any open set. A classical motivation to study the measure $\mu$ comes from Birkhoff's Ergodic Theorem: given any point $x \in S^1$ and any interval $I \subset S^1$ we have that:
\begin{equation}\label{eqfreq}
\lim_{n\to+\infty}\,\frac{1}{n}\,\,\#\big\{j:0 \leq j <n\quad\mbox{and}\quad f^j(x) \in I\big\}=\mu(I)\,.
\end{equation}
Note the following dichotomy: either $\mu$ is \emph{absolutely continuous} with respect to Lebesgue, or it is \emph{singular} (otherwise we would have a decomposition $\mu=\nu_1+\nu_2$, where $\nu_1$ is absolutely continuous, $\nu_2$ is singular and both are non-zero. Being smooth, $f$ preserves the class of zero Lebesgue measure sets, and therefore both $\nu_1$ and $\nu_2$ are $f$-invariant, which contradicts the unique ergodicity of $f$).

In \cite{arnold}, Arnold gave examples (see also \cite[Section I.5]{livrounidim}) of real-analytic circle diffeomorphisms that are minimal, but where the invariant probability measure $\mu$ is \emph{not} absolutely continuous with respect to Lebesgue. In those examples the rotation number is \emph{Liouville}, and any conjugacy with the corresponding rotation maps a set of zero Lebesgue measure onto a set of positive Lebesgue measure. However, in the same work, he proved that any real-analytic diffeomorphism with \emph{Diophantine} rotation number, which is a small perturbation of a rigid rotation, is conjugate to the corresponding rotation by a real-analytic diffeomorphism. Arnold also conjectured that no restriction on how close the diffeomorphism is to a rotation should be necessary. This was later proved by Herman for a large class of Diophantine numbers \cite{hermanihes}, and extended by Yoccoz in \cite{yoccoz2} for all Diophantine numbers, with important contributions by Y.~Katznelson and D.~Ornstein \cite{ko1} (see also \cite{Stark}). Roughly speaking, any $C^{2+\varepsilon}$ diffeomorphism with Diophantine rotation number is conjugate to the corresponding rotation by a $C^1$ diffeomorphism (see \cite[Section I.3]{livrounidim} for precise statements; see also \cite{khaninteplinsky2} and references therein). In these cases, by the Radon-Nikodym theorem, the $\mu$-measure of any given interval is comparable to its Euclidean length. In particular, for smooth diffeomorphisms with Diophantine rotation number, the qualitative notion of minimality can be strengthened to a quantitative one, for \eqref{eqfreq} implies that the \emph{asymptotic frequency} with which any given point visits an open interval is comparable to the Euclidean length of the given interval.

Finally, we remark that the Arnold-Herman-Yoccoz theory described in the previous paragraph also yields the results that any two $C^{\infty}$ circle diffeomorphisms with the same Diophantine rotation number are $C^{\infty}$-conjugate to each other, and that any two real-analytic diffeomorphisms with the same Diophantine rotation number are conjugate to each other by a real-analytic diffeomorphism (again, see \cite[Section I.3]{livrounidim} and references therein). We refer the reader to the survey \cite{surveyYoc} for much more on Yoccoz's seminal contributions to the theory of circle diffeomorphisms.

\section{Multicritical circle maps}\label{seccritmaps}

Let us now define the maps which are the main object of study in the present survey. We start with the notion of {\it non-flat critical point\/}. 

\begin{definition}\label{defnaoflat} We say that a critical point $c$ of a $C^r$ one-dimensional map $f$ is \emph{non-flat} of degree $d>1$ if there exists 
a neighborhood $W$ of the critical point such that $f(x)=f(c)+\phi(x)\,\big|\phi(x)\big|^{d-1}$\, for all $x \in W$, where $\phi : W \rightarrow \phi(W)$ is a $C^r$ diffeomorphism such that $\phi(c)=0$. The number $d$ is also called the \textit{criticality}, the \textit{type} or the \textit{order} of $c$. 
\end{definition}

Note that every critical point of a real-analytic map is non-flat, and its criticality must be a positive integer.

\begin{definition}\label{defmulticritic} A \emph{multicritical circle map} is an orientation preserving $C^3$ circle homeomorphism having $N \geq 1$ critical points, all of which are non-flat.
\end{definition}

Being a homeomorphism, a multicritical circle map $f$ has a well defined rotation number (Theorem \ref{TeoPoinnumrot}). We will assume that $f$ has \emph{irrational} rotation number $\rho\in(0,1)$, in which case it follows from Proposition \ref{lemaunerg} that there exists a \emph{unique} $f$-invariant Borel probability measure $\mu$.

\begin{definition}\label{signature} We define the \emph{signature} of $f$ to be the $(2N+2)$-tuple 
\[
(\rho\,;N;\,d_0,d_1,\ldots,d_{N-1};\,\delta_0,\delta_1,\ldots,\delta_{N-1}),
\]
where $d_i$ is the criticality of the critical point $c_i$ for $0\leq i\leq N-1$, and $\delta_i=\mu[c_i,c_{i+1})$ (with the convention that $c_{N}=c_0$).
\end{definition}

In this section we provide some interesting families of real-analytic circle dynamics.

\subsection{Blaschke products}\label{secBlaschke} Consider the two-parameter family $f_{a,\omega}:\widehat{\C}\to\widehat{\C}$ of \emph{Blaschke products} in the Riemann sphere $\widehat{\C}$ given by:
\begin{equation}\label{blaschke-eq}
f_{a,\omega}(z)=e^{2\pi i\omega}\,z^2\left(\frac{z-a}{1-az}\right)\quad\mbox{for $a \geq 3$ and $\omega\in[0,1)$.}
\end{equation}

Just as any Blaschke product, every map in this family commutes with the geometric involution around the unit circle $\Phi(z)=1/\bar{z}$ (note that $\Phi$ is the identity in the unit circle), and then it leaves invariant the unit circle (Blaschke products \emph{are} the rational maps leaving invariant the unit circle). Moreover, its restriction to $S^1$ is a real-analytic homeomorphism (the fact that $f_{a,\omega}$ has topological degree one, when restricted to the unit circle, follows from the Argument Principle since it has two zeros and one pole in the unit disk). When $a>3$, each $f_{a,\omega}$ has four critical points in the Riemann sphere, which are all different and non-degenerate (quadratic), given by $0$, $\infty$,
\begin{equation}\label{primraiz}
w_a=\frac{a^2+3}{4a}+\frac{\sqrt{(a+3)(a+1)(a-1)(a-3)}}{4a}\,>1\quad\mbox{ and}
\end{equation}
\begin{equation}\label{segraiz}
1/w_a=\frac{a^2+3}{4a}-\frac{\sqrt{(a+3)(a+1)(a-1)(a-3)}}{4a}\,\,\,\in(0,1)\,.
\end{equation}
In particular, the restriction of $f_{a,\omega}$ to the unit circle is a real-analytic diffeomorphism for any $a>3$. When $a \to 3$, both critical points $w_a>1$ and $1/w_a\in(0,1)$ collapse to the point $w=1$, as we can see from \eqref{primraiz} and \eqref{segraiz}. In other words: when $a \to 3$, the family $f_{a,\omega}$ converges to the \emph{boundary} of the space of circle diffeomorphisms: for any $\omega\in[0,1)$, the restriction of $f_{3,\omega}$ to $S^1$ is a real-analytic multicritical circle map with a single critical point at $1$, which is of cubic type, and with critical value $e^{2\pi i\omega}$.

Now let $p,q \in \mathbb{C}$ with $|p|>1$, $|q|>1$, let $\omega\in[0,1)$ and consider $g_{p,q,\omega}:\widehat{\C}\to\widehat{\C}$ given by
\begin{equation}\label{Zakeri}
g_{p,q,\omega}(z)=e^{2\pi i\omega}z^{3}\left(\frac{z-p}{1-\overline{p}z}\right)\left(\frac{z-q}{1-\overline{q}z} \right).
\end{equation}Just as before, every map in this family leaves invariant the unit circle. The following fact was proved by Zakeri in \cite[Section 7]{zak}.

\begin{theorem}\label{TeoZak} For any given $\rho\in(0,1)\!\setminus\!\mathbb{Q}$ and $\delta\in(0,1)$ there exists a unique $g_{p,q,\omega}$ of the form \eqref{Zakeri} such that $g_{p,q,\omega}|_{S^{1}}$ is a bi-critical circle map with signature $(\rho\,;2;3,3;\,\delta,1-\delta)$.
\end{theorem}

\begin{remark}\label{remzak} It would be interesting to extend Zakeri's construction in order to obtain representative families of Blaschke products that restrict to multicritical circle maps with $N \geq 3$ critical points. Such construction should be useful to understand rigidity and renormalization problems for multicritical circle maps with any given number of critical points (see sections \ref{smoothrigid} to \ref{sec:conc}).
\end{remark}

\subsection{The Arnold family}\label{secArnold} Consider the two-parameter family $F_{a,b}:\C\to\C$ of entire maps in the complex plane given by:$$F_{a,b}(z)=z+a-\left(\frac{b}{2\pi}\right)\sin(2\pi z)\quad\mbox{for $a\in[0,1)$ and $b \geq 0$.}$$

\begin{figure}[t]
\begin{center}~
\hbox to \hsize{
\includegraphics[width=5.0in]{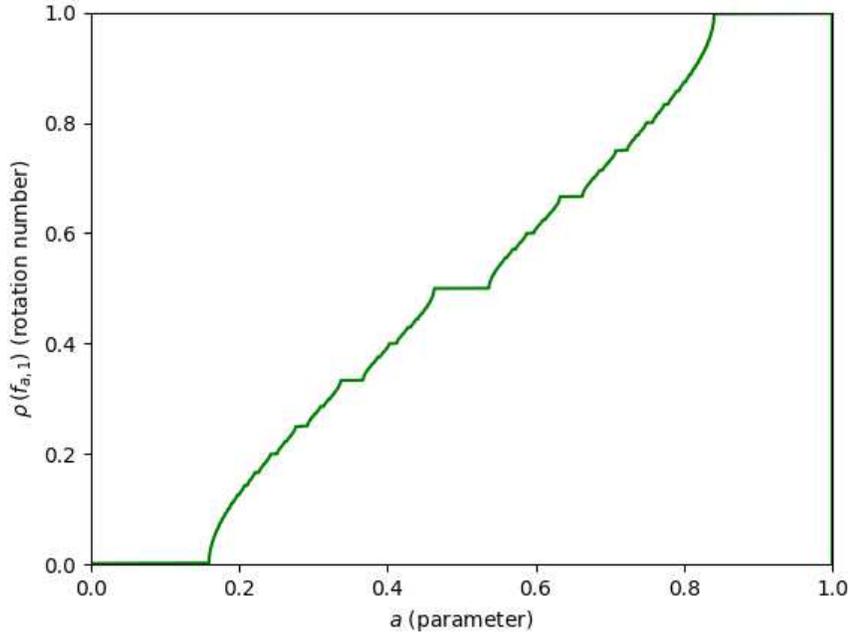}
}
\end{center}
\caption[devilstair]{\label{devilstair}The rotation number, as it varies in the one parameter family $f_{a,1}:\;x\mapsto x+ a - \frac{1}{2\pi}\sin{2\pi x}$, produces a devil staircase.  }
\end{figure}

Since each $F_{a,b}$ commutes with unitary horizontal translation, it is the lift of a holomorphic map of the punctured plane $f_{a,b}:\C\setminus\{0\}\to\C\setminus\{0\}$ under the universal cover $z \mapsto e^{2\pi iz}$. Since $F_{a,b}$ preserves the real axis, $f_{a,b}$ preserves the unit circle. This classical two-parameter family of real-analytic circle maps was introduced by Arnold in \cite{arnold}, and it is known as the \emph{Arnold family}.

For $b=0$, the family $f_{a,b}:S^1 \to S^1$ is just the family of rigid rotations $z \mapsto e^{2\pi ia}z$. As it is easy to check, for $b \in (0,1)$ the Arnold family is still contained in the space of real-analytic circle diffeomorphisms. For $b=1$, however, the Arnold family belongs to the \emph{boundary} of the space of circle diffeomorphisms: each $F_{a,1}$ projects to an orientation preserving real-analytic circle homeomorphism $f_{a,1}$, which presents a critical point (of cubic type) at the point $1$, the projection of the integers. The rotation number of $f_{a,1}$ varies with the parameter $a$ in a continuous, monotone, non-decreasing way, and the resulting graph is an example of what one calls a \emph{devil staircase}; see Figure \ref{devilstair}. Each interval $\big\{a\in[0,1):\rho(f_{a,1})=\theta\big\}$ degenerates to a point whenever $\theta$ is irrational and moreover, the set $\big\{a\in[0,1):\rho(f_{a,1})\in\R\setminus\Q\big\}$ has zero Lebesgue measure \cite{swiatek}. For integers $0 \leq p < q$, the set $\big\{a\in[0,1):\rho(f_{a,1})=p/q\big\}$ is a non-degenerate closed interval. Its interior is made up of critical circle maps with two periodic orbits (both of period $q$), one attracting and one repelling, which collapse to a single parabolic orbit in the boundary of this interval~\cite{epskeentres}. See Item \eqref{itemCvi} in Section \ref{sec:conc} for more on this.

Finally, we remark that for $b>1$ the maps $f_{a,b}:S^1 \to S^1$ are no longer invertible (they present two quadratic critical points). The dynamics of these maps is much richer than the case of homeomorphisms: the rotation number becomes a rotation \emph{interval}, and typical dynamics here have positive topological entropy, infinitely many periodic orbits (coexisting with dense orbits) and, under certain conditions on the combinatorics, they preserve an absolutely continuous probability measure (see \cite{boyland,charles,CGP19,mis} and references therein).

\medskip

The examples presented in both \ref{secBlaschke} and \ref{secArnold} show how multicritical circle maps arise as bifurcations from circle diffeomorphisms to endomorphisms, and in particular, from zero to positive topological entropy (compare with infinitely renormalizable unimodal maps \cite[Chapter VI]{livrounidim}). This is one of the main reasons why multicritical circle maps attracted the attention of physicists and mathematicians interested in the \emph{boundary of chaos} \cite{dgk,feigetal,ksh,lanford1,lanford2,mckay,mackay,ostlundetal,rand1,rand2,rand3, sh}.

\section{No wandering intervals}\label{secyoccozthm} Being a homeomorphism, a multicritical circle map $f$ has a well defined rotation number. Just as before, we will focus on the case when $f$ has no periodic orbits. In the early eighties, J.-C. Yoccoz \cite{yoccoz1} proved that $f$ has no wandering interval. More precisely, we have the following fundamental result. 

\begin{theorem}[Yoccoz]\label{yoccoztheorem} Let $f$ be a multicritical circle map with irrational rotation number $\rho$. 
Then $f$ is topologically conjugate to the rigid rotation $R_{\rho}$, i.e., there exists a homeomorphism $h: S^{1} \rightarrow S^{1}$ such that $h \circ f = R_{\rho} \circ h.$
\end{theorem}

It is not possible to remove the non-flatness condition on the critical points (recall definitions \ref{defnaoflat} and \ref{defmulticritic}): in \cite{hall}, Hall was able to construct $C^\infty$ homeomorphisms of the circle with no periodic points and no dense orbits (see also \cite{pal}).

In the presence of critical points, we should not expect distortion estimates on the first derivative as the ones developed for diffeomorphisms in Section \ref{sec:Denjoy} (we certainly can \emph{not} apply Proposition \ref{1+bv} to a multicritical circle map $f$ since $\log{Df}$ is unbounded, see Figure \ref{logcocycle}). In this section we introduce a new notion of distortion, more suitable for one-dimensional dynamics with non-flat critical points, and we briefly explain the main ideas in the proof of Yoccoz's Theorem \ref{yoccoztheorem} (further applications of this distortion tool to more general no wandering interval problems, mainly developed in \cite{dMvS89,MdMvS}, can be found in \cite[Chapter IV]{livrounidim}).

\begin{figure}[t]
\begin{center}~
\hbox to \hsize{\psfrag{0}[][][1]{$c_i$} \psfrag{S}[][][1]{$S^1$}
\psfrag{1}[][][1]{$\,c_j$}  \psfrag{2}[][][1]{$\ \ \ \ c_k$}
\psfrag{L}[][][1]{$\ \ \ \ \ \ \log{Df}$} 
\hspace{1.0em} \includegraphics[width=5.0in]{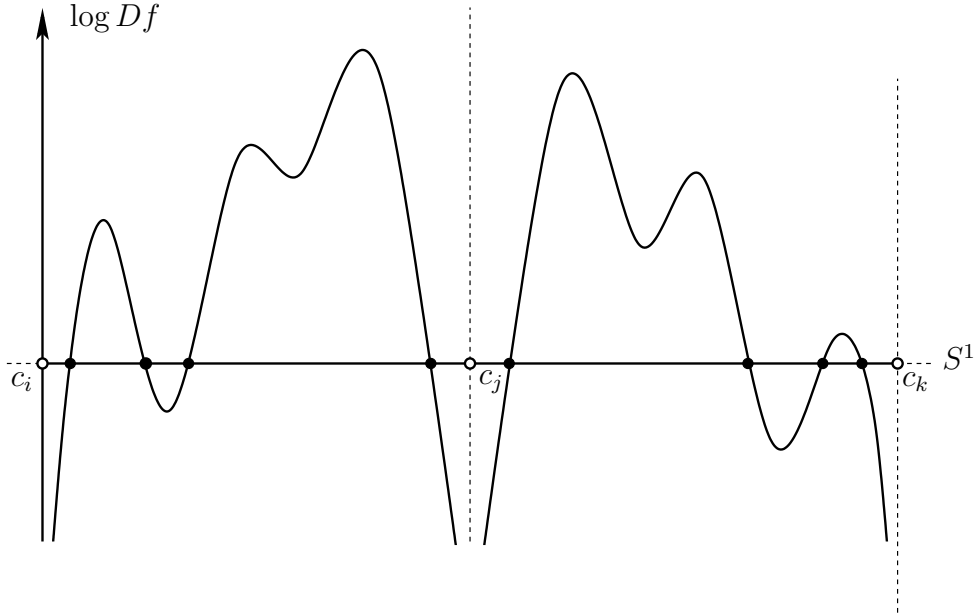}
}
\end{center}
\caption[logcocycle]{\label{logcocycle} The cocycle $\log{Df}$ is unbounded for a multicritical circle map $f$.}
\end{figure}

Given two intervals $M\subset T\subset S^{1}$ with $M$ compactly contained in $T$ (written $M\Subset T$) let us denote by $L$ and $R$ the two connected components of $T\setminus M$. We define the \emph{cross-ratio} of the pair $M,T$ as follows:
\begin{equation*}
[M,T]= \frac{|L|\,|R|}{|L\cup M|\,|M \cup R|} \in (0,1),
\end{equation*}where $|I|$ denotes the Euclidean length of an interval $I$. Write $M=(b,c)$ and $T=(a,d)$, and let $\phi$ be the M\"obius transformation determined by $\phi(a)=0$, $\phi(c)=1$ and $\phi(d)=\infty$. Then$$[M,T]=\phi(b)=\left(\frac{d-c}{c-a}\right)\left(\frac{b-a}{d-b}\right).$$

The \textit{cross-ratio distortion} of any element $f\in\home(S^1)$ on the pair of intervals $(M,T)$ is defined to be the ratio
\begin{equation*}
\crd(f;M,T)= \frac{\big[f(M),f(T)\big]}{[M,T]}\,.
\end{equation*}
By definition, the cross-ratio is preserved by M\"obius transformations.

\medskip

The main tool available to control cross-ratio distortion around non-flat critical points is the \textit{Schwarzian derivative}, which is the differential operator defined for all $x$ regular point of $f$ by:
 \begin{equation*}
  Sf(x)= \dfrac{D^{3}f(x)}{Df(x)} - \dfrac{3}{2} \left( \dfrac{D^{2}f(x)}{Df(x)}\right)^{2}.
 \end{equation*}

The kernel of the Schwarzian derivative is the group of M\"obius transformations\footnote{The fact that the Schwarzian derivative vanishes at M\"obius transformations 
is a straightforward computation. On the other hand, given an increasing $C^3$ diffeomorphism $f$, consider $g=(Df)^{-1/2}$ and note that $Sf=-2\,D^2g/g$. If $f$ has zero Schwarzian derivative then $g$ is affine, which implies at once that $f$ is a M\"obius transformation.}. In particular, the cross-ratio is preserved by maps with zero Schwarzian derivative. As it turns out, it is weakly contracted by maps with negative Schwarzian derivative:

\begin{lemma}\label{contracts} If $f$ is a $C^3$ diffeomorphism with $Sf<0$, then for any two intervals $M\subset T$ contained in the domain
of $f$ we have $\crd(f;M,T)<1$, that is, $\big[f(M),f(T)\big]<[M,T]$.
\end{lemma}

A straightforward computation shows that the Schwarzian derivative is always negative around a non-flat critical point. More precisely, we have the following fact.

\begin{lemma}\label{negsch} Given $f$ with a non-flat critical point $c$, there exist a neighbourhood $U$ of $c$ and a constant $K=K(f)>0$ such that for all $x \in U \setminus \{ c \}$ we have$$Sf(x)< -\,\frac{K}{(x-c)^2}\,.$$
\end{lemma}

Proofs of these basic facts can be found, for instance, in the appendix of \cite{EdFG}. Finally, given a family of intervals $\mathcal{F}$ on $S^1$ and a positive integer $m$, we say that $\mathcal{F}$ has {\it multiplicity of intersection at most $m$\/} if each $x \in S^1$ belongs to at most $m$ elements of $\mathcal{F}$.

\begin{theorem}\label{CRIswiatek} Given a multicritical critical circle map $f:S^1\to S^1$, there exists a constant $C>1$, depending only on $f$, such that the following holds. If $M_i\Subset T_{i} \subset S^1$, where $i$ runs through some finite set of indices $\mathcal{I}$, are intervals on the circle such that the family $\{T_i: i\in \mathcal{I}\}$ has multiplicity of intersection at most $m$, then
\begin{equation}\label{crossprod}
\prod_{i \in \mathcal{I}} \crd(f;M_{i},T_{i}) \leq C^{m}\ .
\end{equation}
\end{theorem}

Theorem \ref{CRIswiatek}, usually named \emph{The Cross-Ratio Inequality}, was first obtained by Yoccoz in a slightly different form \cite[Section 4]{yoccoz1}. The specific version stated above can be found in \cite[Section 2]{swiatek}. Roughly speaking, its proof goes as follows: let $\mathcal{U}=\bigcup W_i$, where the $W_i$'s are as in Definition \ref{defnaoflat}, and let $\mathcal{V}$ be an open set with $\mathcal{U}\cup \mathcal{V}=S^1$ whose closure does not contain any critical point of $f$.
We assume without loss of generality that the maximum length of the $T_i$'s is smaller than the Lebesgue number of the covering 
$\{\mathcal{U},\mathcal{V}\}$. 
Write the product on the left-hand side of \eqref{crossprod} as $P_1\cdot P_2$, where 
\[
 P_1\;=\;\prod_{T_i\subseteq \mathcal{V}}  \crd(f;M_{i},T_{i})\ \ \ ,\ \ \ P_2\;=\;\prod_{T_i\subseteq \mathcal{U}}  \crd(f;M_{i},T_{i})\ .
\]
Then on the one hand we claim that $P_1\leq e^{2mV}$, where $V=\varia(\log{Df}|_{\mathcal{V}})$. Indeed:
\begin{align}\label{estCrDBV}
\log P_1&=\sum_{T_i\subseteq \mathcal{V}}\log\crd(f;M_{i},T_{i})\\\notag
&=\sum_{T_i\subseteq \mathcal{V}}\log Df(w_i)-\log Df(x_i)+\log Df(y_i)-\log Df(z_i) \leq 2mV\,,
\end{align}where the points $w_i$, $x_i$, $y_i$ and $z_i$ belong to $T_i$ and are given by the Mean Value Theorem. On the other hand, the factors making up $P_2$ are of two types: those such that $f|_{T_i}$ is a diffeomorphism onto its image, and those such that $T_i$ contains some critical point of $f$. By Lemma \ref{negsch}, all factors of the first type are diffeomorphisms with negative Schwarzian and therefore, by Lemma \ref{contracts}, satisfy $\crd(f;M_i,T_i)<1$. Factors of the second type are easily controlled by the non-flatness condition: $\crd(f;M_{i},T_{i}) \leq 9d^2$, where $d>1$ is the order of the critical point that belongs to $T_i$ (see for instance \cite[Lemma 2.2(4)]{EdFG}). Since there are at most $mN$ such factors (where $N$ is the number of critical points of $f$), the result follows. For more details, see \cite[Section 2]{swiatek}.

Theorem \ref{CRIswiatek} is a great tool to estimate cross-ratio distortion of large iterates of multicritical circle maps. Indeed, note first that given two intervals $M\Subset T$ and given $n\in\nt$ we have the following \emph{chain rule}:$$\crd(f^{n};M,T)=\prod_{i=0}^{n-1}\crd\big(f;f^i(M),f^i(T)\big).$$With this at hand, Theorem \ref{CRIswiatek} can be used to prove that \emph{dynamically symmetric} intervals are \emph{comparable}. More precisely, we have the following fact, which is \cite[Lemma 3.3]{EdF}.

\begin{lemma}\label{lemmasymint} Given a multicritical circle map $f$ there exists a constant $C> 1$ depending only on $f$ such that, for all $n\geq 0$ and all $x \in S^{1}$, we have$$C^{-1}\,\big|x- f^{-q_{n}}(x)\big| \leq \big|f^{q_{n}}(x)-x\big| \leq C\,\big|x- f^{-q_{n}}(x)\big|\,.$$
\end{lemma}

A proof of Lemma \ref{lemmasymint} can be found in \cite[Section 3.1]{EdF}. Now we claim that Theorem \ref{yoccoztheorem} is a straightforward consequence of Lemma \ref{lemmasymint}. Indeed, assume, by contradiction, that there exists a wandering interval $J=(a,b)$. For each $n\in\nt$, consider the dynamically symmetric intervals $L_n=\big(f^{-q_n}(a),a\big)$ and $R_n=\big(a,f^{q_n}(a)\big)$. Since we are assuming that $J$ is a wandering interval, $R_n$ contains $J$ for all $n\in\nt$, and then we deduce from Lemma \ref{lemmasymint} that the sequence $\big\{|L_n|\big\}_{n\in\nt}$ is bounded away from zero. However, if we assume that $J$ is a \emph{maximal} wandering interval, the point $a$ is recurrent (both for the past and the future) and then $\big\{|L_n|\big\}_{n\in\nt} \to 0$ as $n$ goes to infinity. This contradiction shows that no such wandering interval $J$ exists, finishing the proof of Theorem~\ref{yoccoztheorem}.

\begin{remark} The argument above gives a new proof of Denjoy's Theorem \ref{denjoy}, since the Cross-Ratio Inequality (Theorem \ref{CRIswiatek}) certainly holds true whenever $f$ is a $C^1$ diffeomorphism and $\log Df$ has bounded variation (note that the Schwarzian derivative is not needed in this case: estimate \eqref{estCrDBV} holds on the whole circle).
\end{remark}

\section{Real bounds}\label{secrealbounds}

\begin{figure}[b]
\begin{center}~
\hbox to \hsize{
\psfrag{0}[][][1]{$\scriptstyle{x}$} 
\psfrag{1}[][][1]{$\scriptstyle{f(x)}$}
\psfrag{2}[][][1]{$\!\scriptstyle{f^2(x)}$}
\psfrag{3}[][][1]{$\;\;\scriptstyle{f^3(x)}$}
\psfrag{4}[][][1]{$\scriptstyle{f^4(x)}$}
\psfrag{5}[][][1]{$\scriptstyle{f^5(x)}$}
\psfrag{6}[][][1]{$\scriptstyle{f^6(x)}$}
\psfrag{I0}[][][1]{$I_2$}
\psfrag{I1}[][][1]{$f(I_2)$}
\psfrag{J0}[][][1]{$I_1$}
\psfrag{J1}[][][1]{$f(I_1)$}
\psfrag{J2}[][][1]{$f^2(I_1)$}
\psfrag{J3}[][][1]{$f^3(I_1)$}
\psfrag{J4}[][][1]{$f^4(I_1)$}
\includegraphics[width=5.6in]{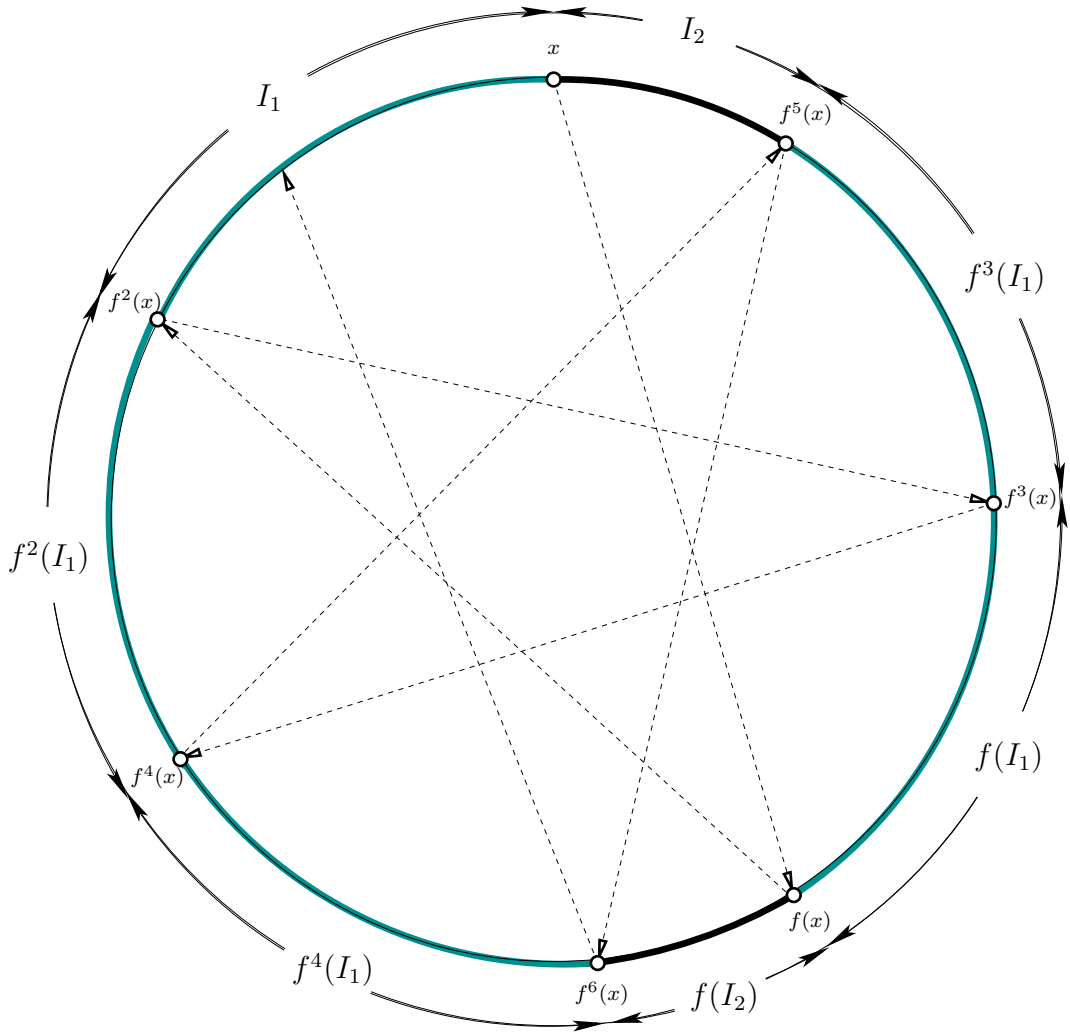}
   }
\end{center}
\caption[dynpar]{\label{dynpar} Dynamical partition $\mathcal{P}_1(x)$
of a circle homeomorphism with rotation number $\rho(f)=\sqrt{2}-1=[2,2,2,\ldots]$.}
\end{figure}

Let $f$ be a $C^3$ multicritical circle map with irrational rotation number $\rho\in(0,1)$, and let us fix some base point $x \in S^1$. For each non-negative integer $n$, let $I_{n}$ be the interval with endpoints $x$ and $f^{q_n}(x)$ containing  $f^{q_{n+2}}(x)$, namely, $I_n=\big[x,f^{q_n}(x)\big]$ and $I_{n+1}=\big[f^{q_{n+1}}(x),x\big]$ (recall that the sequence of return times $\{q_n\}$ has been constructed in Section~\ref{secdifeos}). As it is not difficult to prove, for each $n\geq 0$, the collection of intervals$$\mathcal{P}_n(x)\;=\; \left\{f^i(I_n):\;0\leq i\leq q_{n+1}-1\right\}\cup\left\{f^j(I_{n+1}):\;0\leq j\leq q_{n}-1\right\}$$is a partition (modulo endpoints) of the unit circle (see for instance the appendix in \cite{EdFG}), called the {\it $n$-th dynamical partition\/} associated to $x$. The intervals of the form $f^i(I_n)$ are called \emph{long}, whereas those of the form $f^j(I_{n+1})$ are called \emph{short}. The initial partition $\mathcal{P}_0(x)$ is given by$$\mathcal{P}_0(x)=\left\{\big[f^{i}(x),f^{i+1}(x)\big]:\,i\in\{0,...,a_0-1\}\right\}\cup\big\{\big[f^{a_{0}}(x),x\big]\big\},$$where $a_0$ is the integer part of $1/\rho$ (see Section \ref{secdifeos}).

\begin{example} Figure \ref{dynpar} shows the dynamical partition $\mathcal{P}_1(x)$ associated to a circle homeomorphism with rotation number $\rho(f)=\sqrt{2}-1=[2,2,2,\ldots]$, for which $q_1=2$ and $q_2=5$. Thus, explicitly, $\mathcal{P}_1(x)=\left\{I_1,f(I_1),f^2(I_1),f^3(I_1),f^4(I_1)\right\} \cup \left\{I_2,f(I_2)\right\}$. 
\end{example}

The following fundamental result was obtained by Herman and \'Swi\k{a}tek in the eighties \cite{H,swiatek}.

\begin{theorem}[Real Bounds]\label{realbounds} Given $N\geq 1$ in $\nt$ and $d>1$ there exists a universal constant $C=C(N,d)>1$ with the following property: for any given multicritical circle map $f$ with irrational rotation number, and with at most $N$ critical points whose criticalities are bounded by $d$, there exists $n_0=n_0(f)\in\nt$ such that for each critical point $c$ of $f$, for all $n \geq n_0$, and for every pair $I,J$ of adjacent atoms of $\mathcal{P}_n(c)$ we have:$$C^{-1}\,|I| \leq |J| \leq C\,|I|\,,$$where $|I|$ denotes the Euclidean length of an interval $I$.
\end{theorem}

Note that for a rigid rotation we have\, $|I_n|=a_{n+1}\,|I_{n+1}|+|I_{n+2}|$. If $a_{n+1}$ is big, then $I_n$ is much larger than $I_{n+1}$. Thus, even for rigid rotations, real bounds do not hold in general.

The proof of Theorem \ref{realbounds} follows by combining the Cross-Ratio Inequality (Theorem \ref{CRIswiatek}) and Lemma \ref{lemmasymint} with a careful combinatorial analysis that the reader can find in \cite[Sections 3.2-3.3]{EdF}. To obtain bounds independent of $f$ in the statement of Theorem \ref{realbounds}, we must replace Theorem \ref{CRIswiatek} by the following result, obtained in \cite[Theorem B]{EdFG}.

\begin{theorem}\label{CRI} Given $N\geq 1$ in\, $\nt$ and $d>1$ there exists a constant $B=B(N,d)>1$ with
the following property: given a multicritical circle map $f$, with at most $N$ critical points whose criticalities are bounded by $d$, there exists $n_0=n_0(f)$ such that for all $n \geq n_0$, $\Delta\in\mathcal{P}_n(c)$ and $k\in\nt$ such that $f^i(\Delta)$ is contained in an
element of $\mathcal{P}_n(c)$ for all $1\leq i \leq k$, we have that$$\crd(f^k;\Delta,\Delta^*)\leq B\,,$$where $\Delta^*$ denotes the union of $\Delta$ with its left and right neighbours in $\mathcal{P}_n(c)$.
\end{theorem}

See \cite{EdF,EdFG} for more details. For more on the relation between the Schwarzian derivative and cross-ratio distortion we refer the reader to \cite[Chapter IV]{livrounidim} and to \cite[Chapter 6]{dFdMbook}, where a proof of Theorem \ref{realbounds} is also sketched.

\medskip

We finish Section \ref{secrealbounds} by pointing out that it is essential for Theorem \ref{realbounds} to be true to consider dynamical partitions \emph{around critical points}. Indeed, let us agree to say that $f$ has \emph{bounded geometry} at $x \in S^1$ if there exists $K>1$ such that for all $n\in\nt$ and for every pair $I,J$ of adjacent atoms of $\mathcal{P}_n(x)$ we have$$K^{-1}\,|I| \leq |J| \leq K\,|I|\,.$$Following \cite[Section 1.4]{dFG2019}, we consider the set$$\mathcal{A}=\mathcal{A}(f)=\{x \in S^1:\,\mbox{$f$ has bounded geometry at $x$}\}\,.$$As explained in \cite[Section 1.4]{dFG2019}, the set $\mathcal{A}$ is $f$-invariant. Moreover, as stated in Theorem \ref{realbounds}, all critical points of $f$ belong to $\mathcal{A}$. Being $f$-invariant and non-empty, the set $\mathcal{A}$ is dense in the unit circle. However, even in the case of maps with a single critical point, $\mathcal{A}$ can be rather small. Indeed, the following is \cite[Theorem D]{dFG2019}.

\begin{theorem}\label{teoAmagro} There exists a full Lebesgue measure set $\bm{R}\subset(0,1)$ of irrational numbers with the following property: let $f$ be a $C^3$ critical circle map with a single (non-flat) critical point and rotation number $\rho\in\bm{R}$. Then the set $\mathcal{A}(f)$ is meagre (in the sense of Baire) and it has zero $\mu$-measure (where $\mu$ denotes the unique $f$-invariant probability measure).
\end{theorem}

As it is not difficult to prove, the full Lebesgue measure set $\bm{R}\subset(0,1)$ given by Theorem \ref{teoAmagro} contains no bounded type numbers (see \cite{dFG2019} for much more).

\section{Some ergodic aspects}\label{secerg}

In this section we examine multicritical circle maps from the point of view of measurable dynamics. 

\subsection{No $\sigma$-finite measures} As it was proved in Proposition \ref{lemaunerg}, a $C^3$ multicritical circle map $f$ with irrational rotation number is uniquely ergodic. Moreover, as we saw in Section \ref{secyoccozthm}, a theorem due to J.-C. Yoccoz (Theorem \ref{yoccoztheorem}) asserts that $f$ is minimal and therefore topologically conjugate to the corresponding rigid rotation. Yoccoz's theorem implies that the support of the unique invariant Borel probability measure of a multicritical circle map with irrational rotation number is the whole circle. As it turns out, this measure is \emph{purely singular} with respect to Lebesgue measure. This result was proved by Khanin in the late eighties, by means of a certain thermodynamic formalism \cite[Theorem 4]{khanin91} (see also \cite[Proposition 1]{GS93}). In fact, in \cite{dFG2020} the authors have proved a more general result (namely, Theorem \ref{teomedidadFG2020} below), which we proceed to state. We need a definition and an auxiliary result. By a \emph{nested sequence of partitions} we understand a sequence $\{\mathcal{Q}_n\}_{n\geq 0}$ of finite interval partitions of $S^1$ (modulo endpoints) which is nested in the sense that each atom of $\mathcal{Q}_n$ is a union of atoms of $\mathcal{Q}_{n+1}$, for all $n\geq 0$, and is such that $\mathrm{mesh}(\mathcal{Q}_n)\to 0$ as $n\to \infty${\footnote{Here, as usual, the \emph{mesh} of a partition is the maximum length of its atoms.}}.

\begin{definition}\label{katzdef}
 A $C^1$ circle homeomorphism $f$ has the \emph{Katznelson property} if there exist
 a nested sequence of partitions $\{\mathcal{Q}_n\}_{n\geq 0}$ and constants $\alpha, \beta>1$ for which the following is true for all $n\geq 0$. For every atom $\Delta\in \mathcal{Q}_n$, there exist two disjoint closed subintervals $J_1,J_2\subset \Delta$ and a positive integer $k$ such that (i) $|J_1|\geq \alpha |J_2|$; (ii) we have $f^k(J_1)=J_2$, and (iii) $f^k:J_1\to J_2$ is diffeomorphism whose distortion is bounded by $\beta$.
\end{definition}

\begin{theorem}\label{katzcrit}
 Let $f:S^1\to S^1$ be a $C^1$ minimal homeomorphism, and suppose that $f$ has the Katznelson property. Then $f$ does not admit a $\sigma$-finite invariant measure which is absolutely continuous with respect to Lebesgue measure. 
\end{theorem}

For a proof of this result, see \cite[Th.~3.1]{dFG2020}. In that paper, we refer to this theorem  as the \emph{Katznelson criterion}, because it is a slight generalization of a result proved by Katznelson in \cite{katz1}.

\begin{theorem}\label{teomedidadFG2020} Let $f:S^1\to S^1$ be a $C^3$ multicritical circle map with irrational rotation number. Then $f$ does not admit a $\sigma$-finite invariant measure which is absolutely continuous with respect to Lebesgue measure. 
\end{theorem}

The strategy to prove this theorem is, not surprisingly, to show that $f$ possesses the Katznelson property, and then invoke Theorem \ref{katzcrit}. In veryfying Katznelson's property, one can restrict attention to the nested sequence of partitions given by the dynamical partitions $\mathcal{P}_n(c)$ of some critical point of $f$ (recall Section \ref{secrealbounds}). The required geometric control comes from the real bounds (Theorem \ref{realbounds}), plus  the negative Schwarzian property of first return maps. When the rotation number of $f$ is of unbounded type, it is also necessary to use Yoccoz's inequality to control almost parabolic returns (see Lemma \ref{lemyoccoz} below). A detailed proof of Theorem \ref{teomedidadFG2020} can be found in \cite[\S\,4]{dFG2020}. 

\subsection{Lyapunov exponents} As mentioned in Section \ref{sec:Denjoy}, the analogous statement to Lemma \ref{zeroexpdifeo} also holds for multicritical circle maps. More precisely, we have the following result.

\begin{theorem}[Zero Lyapunov exponent]\label{expzeroccm} Let $f:S^1 \to S^1$ be a $C^3$ multicritical circle map with irrational rotation 
number, and let $\mu$ be its unique invariant Borel probability measure.
Then $\log Df$ belongs to $L^1(\mu)$ and:$$\int_{S^1}\!\log Df\,d\mu=0\,.$$
\end{theorem}

The proof of Theorem \ref{expzeroccm} is considerably more difficult than the one of Lemma \ref{zeroexpdifeo}, since in this case $\log Df$ is not a continuous function: it is defined only in the complement of the critical set of $f$, and it is unbounded (see Figure \ref{logcocycle}). We refer the reader to \cite{dFG2016}, where Theorem \ref{expzeroccm} is obtained as a consequence of the real bounds (Theorem \ref{realbounds}). 

\subsection{Hausdorff dimension} As we have seen above, the unique invariant probability measure $\mu$ under a multicritical circle map $f$ is purely singular with respect to Lebesgue measure. The \emph{Hausdorff dimension} of $\mu$, denoted $\dim_{H}(\mu)$, is by definition the smallest of the Hausdorff dimensions of all measurable sets having full $\mu$-measure . More precisely, 
\[
 \dim_{H}(\mu) = \inf\,\{\dim_{H}(E)\,:\; E\subset S^1\ \text{is measurable and}\ \mu(E)=1\}\ .
\]
A natural question to ask is: how does the Hausdorff dimension of $\mu$ vary with $f$? Obviously, a priori it should only depend on the bi-Lipschitz conjugacy class of $f$. In a recent paper \cite{trujillo}, Trujillo establishes lower and upper bounds for $\dim_{H}(\mu)$ that depend only on the Diophantine nature of the rotation number of $f$. In order to state his result, we first recall that an irrational number $\alpha$ is said to be \emph{Diophantine of type $\omega\geq 0$} if there exists a positive constant $C=C(\alpha)>0$ such that
\[
 \left|\alpha - \frac{p}{q}\right| \;\geq\; \frac{C}{q^{2+\omega}}\ ,\ \ \text{for all}\ p,q\in \mathbb{Z},\ q\neq 0\ .
\]
We denote by $\mathcal{D}_\omega$ the set of all Diophantine numbers of type $\omega$. It is a well-known fact that $\mathcal{D}_0$ has zero Lebesgue measure{\footnote{The set $\mathcal{D}_0$ is precisely the set of numbers of \emph{bounded type}, as previously defined.}}, whereas for each $\omega>0$ the set $\mathcal{D}_\omega$ has full Lebesgue measure. 

\begin{theorem}\label{hausdorffmeasure}
 If $f$ is a $C^3$ multicritical circle map with irrational rotation number $\rho=\rho(f)$ and $\mu$ is its unique invariant probability measure, then the following holds.
 \begin{enumerate}
  \item[(i)] If $\rho\in \mathcal{D}_\omega$ for some $\omega\geq 0$, then there exists $\nu>0$ such that
  \[
   \dim_{H}(\mu)\;\geq\; \frac{1}{2\omega + \nu}\ .
  \]
 \item[(ii)] If $\rho\notin \mathcal{D}_\omega$ for some $\omega>0$, then 
 \[
  \dim_{H}(\mu)\;\leq\; \frac{1}{\omega + 1}\ .
 \]

 \end{enumerate}

\end{theorem}

Note that the above theorem does not provide an upper bound for $\dim_{H}(\mu)$ in the case when the rotation number $\rho(f)$ is of bounded type ({i.e.,} lies in $\mathcal{D}_0$). Nevertheless, it had already been known since the work of Graczyk and Swiatek \cite{GS93} that in the bounded type case the Hausdorff dimension of $\mu$ lies strictly between $0$ and $1$. 

\section{Quasi-symmetric rigidity}\label{secQS}

An important first step towards \emph{smooth rigidity} of multicritical circle maps (to be examined in Section \ref{smoothrigid}) is to answer the question: when are two topologically conjugate multicritical circle maps \emph{quasi-symmetrically} conjugate? 

\subsection{Quasi-symmetry} Recall that an orientation-preserving homeomorphism of $S^1=\mathbb{R}/\mathbb{Z}$, say $h: S^1\to S^1$, is said to be \emph{quasi-symmetric} if there exists a constant $K\geq 1$ such that 
\begin{equation}\label{qsdef}
 \frac{1}{K}\;\leq\;\frac{h(x+t)-h(x)}{h(x)-h(x-t)}\;\leq\; K\ ,\ \ \text{for all}\ x\in S^1\,\text{and all}\  t>0\ . 
\end{equation}
If $K$ is such that $h$ satisfies \eqref{qsdef} for this $K$, then we say that $h$ is \emph{$K$-quasi-symmetric}. The smallest $K$ with this property is called the \emph{quasi-symmetric distortion} of $h$.

Quasi-symmetric homeomorphisms arise as boundary values of \emph{quasi-conformal} homeomorphisms of the unit disk (see \cite[ch.~4]{ahlfors}).

Let us describe a criterion for quasi-symmetry that is particularly useful in the study of circle maps. In order to formulate it, we first need a definition, reproduced almost verbatim from \cite[Def.~5.1]{EdF}.

\begin{definition}\label{finegrid}
 A \emph{fine grid} is a nested sequence $\{\mathcal{Q}_n\}_{n\geq 0}$ of finite interval partitions of the circle (modulo endpoints) having the following properties.
 \begin{enumerate}
  \item[(a)] Each $\mathcal{Q}_{n+1}$ is a strict refinement of $\mathcal{Q}_n$.
  \item[(b)] There exists an integer $a\geq 2$ such that each atom $\Delta \in \mathcal{Q}_n$ is the disjoint union of at most $a$ atoms of $\mathcal{Q}_{n+1}$.
  \item[(c)] There exists $\sigma > 1$ such that $\sigma^{-1}|\Delta| \leq  |\Delta'| \leq  \sigma |\Delta|$  for each pair of adjacent atoms $\Delta, \Delta'\in \mathcal{Q}_n$.
 \end{enumerate}
The numbers $a, \sigma$ are called \emph{fine grid constants}.
\end{definition}

This notion of fine grid was first introduced in \cite[\S 4]{edsonwelington1}. Its usefulness lies in the fact that one can sometimes tell how regular a homeomorphism is by looking at the effect it has on a suitable fine grid. As an example, here is the criterion for quasi-symmetry that we promised above.

\begin{prop}\label{qsfinegrid}
 Let $\{\mathcal{Q}_n\}_{n\geq 0}$ be a fine grid in $S^1$ whose fine grid constants are $a, \sigma$,
and let $h : S^1 \to S^1$ be an orientation-preserving homeomorphism such that
\[
 \left| \frac{|h(\Delta')|}{|h(\Delta'')|} - 
 \frac{|\Delta'|}{|\Delta''|}\right|\;\leq\; \lambda\ ,
\]
for each pair of adjacent atoms $\Delta', \Delta''\in \mathcal{Q}_n$, for all $n\geq 0$, where $\lambda$ is a positive constant. Then there exists $K = K(a, \sigma, \lambda) > 1$ such that $h$ is $K$-quasisymmetric.
\end{prop}

A proof of this result can be found in \cite[pp. 5611--5612]{EdF}.

Now, let us agree to say that an orientation-preserving homeomorphism $h:S^1\to S^1$ is a \emph{fine grid isomorphism} if it maps fine grids to fine grids. 
Then the  criterion for quasi-symmetry given by Proposition \ref{qsfinegrid} has the following consequence.  

\begin{coro}\label{corogrid}
Let $h:S^1\to S^1$ be an orientation-preserving homeomorphism. Then the following are equivalent.
\begin{enumerate}
 \item[(i)] $h$ is quasi-symmetric;
 \item[(ii)] $h$ maps \emph{some} fine grid onto another fine grid;
 \item[(iii)] $h$ is a fine grid isomorphism. 
\end{enumerate}
\end{coro}

As we shall see in the sequence, the characterization of quasi-symmetry provided by Corollary \ref{corogrid} is extremely helpful in the study of critical circle maps. 

\subsection{Quasi-symmetric conjugacies}\label{sec:qsconj}
The facts presented above allow us to give a short proof of the following theorem, originally due to Herman \cite{H}. In this section, since we will consider dynamical partitions associated to different maps, we shall use the notation $\mathcal{P}_n(x,f)$, $I_n(x,f)$, instead of $\mathcal{P}_n(x)$, $I_n(x)$, etc. to emphasize the dependency on $f$.

\begin{theorem}\label{qsherman} A multicritical circle map without periodic points is quasi-symmetrically conjugate to a rigid rotation if and only if its rotation number is of bounded type. 
\end{theorem}

\begin{proof} Let us first assume that $f: S^1\to S^1$ is a multicritical circle map whose rotation number $\rho=[a_0,a_1,a_2,\ldots]$ is an irrational of bounded type, say $a_n\leq A$ for all $n$. Let $c\in S^1$ be a critical point of $f$. 
 We claim that   
 the dynamical partitions $\mathcal{P}_{2n}(c,f)$, $n\geq 0$ constitute a fine grid. 
 Indeed, every atom of $\mathcal{P}_{2n}(c,f)$ is partitioned into at least $2$ and at most $(a_{2n+1}+1)(a_{2n+2}+1)\leq (A+1)^2$ atoms of $\mathcal{P}_{2n+2}(c,f)$, and these are all comparable by Theorem \ref{realbounds}. Hence conditions (a), (b) and (c) of Definition \ref{finegrid} are met, and the claim is proved. Let $h:S^1\to S^1$ be a topological conjugacy between $f$ and the rigid rotation $R_\rho$ (say $h\circ f=R_\rho\circ h$), which exists by Yoccoz's theorem. Then one can easily check that the dynamical partitions $\mathcal{P}_{2n}(h(c),R_\rho)$, $n\geq 0$, also constitute a fine grid (for $R_\rho$). But then $h$ satisfies property (ii) of Corollary \ref{corogrid}, and therefore it must be quasi-symmetric.
 
 For the converse, suppose $h:S^1\to S^1$ is a homeomorphism satisfying $h\circ f = R_\rho\circ h$, and suppose the rotation number of $f$ is not of bounded type. Then there exists a subsequence $(n_i)$ with $a_{n_i+1}\to \infty$ as $i\to \infty$. Again we take $c$ to be a critical point of $f$, and let $x=h(c)$. By the real bounds, the scaling ratios $|I_{n_i+1}(c,f)|/|I_{n_i}(c,f)|$ for $f$ remain bounded, whereas for the rigid rotation we have
 \[
  \frac{|h(I_{n_i+1}(c,f))|}{|h(I_{n_i}(c,f))|}\;=\; 
  \frac{|I_{n_i+1}(x,R_\rho)|}{|I_{n_i}(x,R_\rho)|}\;>\; a_{n_i+1}\;\to\; \infty \ \ \text{as}\ i\to \infty \ ,
 \]
and therefore $h$ cannot be quasi-symmetric.
\end{proof}

\begin{remark}
 In the above proof, the only reason we did not use the full collection of dynamical partitions as our fine grid is that $\mathcal{P}_{n+1}(c,f)$ is \emph{not} a strict refinement of $\mathcal{P}_n(c,f)$ (the \emph{short} atoms of $\mathcal{P}_n(c,f)$ are not decomposed at all in the next step; they become \emph{long} atoms of $\mathcal{P}_{n+1}(c,f)$). This is why we skipped every other level. 
\end{remark}

\begin{remark} An interesting application of Theorem \ref{qsherman} to holomorphic dynamics goes as follows. In the complex quadratic family $P_\theta:\,z\mapsto e^{2\pi i\theta}z + z^2$, one knows that for each Diophantine $\theta$ the fixed point at the origin is linearizable, so it belongs to the Fatou set of $P_\theta$. The component of the Fatou set containing $0$ is a \emph{Siegel disk}; call it $\Omega_\theta$. In \cite{Do}, Douady proved that if $\theta$ is a number of bounded type, then $\partial\Omega_\theta$ is a \emph{quasi-circle} that contains the critical point of $P_\theta$. The rough idea is to start with a Blaschke product $B$ from the family introduced in section \ref{secBlaschke} (see eq.~\eqref{blaschke-eq}) whose restriction to $S^1$ is a critical circle map $f$ with rotation number $\theta$. Then, using Theorem \ref{qsherman}, one applies quasi-conformal surgery to $B$, cutting out the unit disk and glueing it back in using as sewing map the quasi-symmetric conjugacy $h$ between $f$ and the rigid rotation with the same rotation number. Redefining the map in the interior of the unit disk to be that same rotation, and applying the measurable Riemann mapping theorem, the unit circle is mapped onto a quasi-circle, and the post-surgery map becomes $P_\theta$, thereby proving Douady's theorem. This result was later generalized by Petersen and Zakeri in \cite{PZ}. Their theorem allows the rotation number $\theta$ to belong to a certain class of unbounded type numbers, and the proof is accomplished through the use of \emph{trans-quasiconformal surgery}. 
\end{remark}

More important for our purposes in the present survey is the following immediate consequence of Theorem \ref{qsherman}. 

\begin{coro}\label{qsbdtype}
Any two multicritical circle maps $f$ and $g$ with the same irrational rotation number \emph{of bounded type} are quasisymmetrically conjugate, and in fact \emph{every} topological conjugacy between $f$ and $g$ is a quasi-symmetric homeomorphism.
\end{coro}

Note that in Corollary \ref{qsbdtype} the number of critical points of $f$ and the number of critical points of $g$ need not be the same!
But the bounded type hypothesis on the rotation number is essential. 
In full generality, the above statement is most definitely \emph{false} for unbounded combinatorics; see Section \ref{sec:orbflex}. 

What can be said, then, for arbitrary irrational rotation numbers? If $f$ and $g$ have the same number of critical points and there is a conjugacy between $f$ and $g$ that maps each critical point of $f$ to a critical point of $g$, the first part of the statement of Corollary \ref{qsbdtype} continues to hold. More precisely, we have the following result, which is the main theorem in \cite{EdF}.

\begin{theorem}\label{qsrigidity}
 Let $f, g : S^1 \to S^1$ be two $C^3$ multicritical circle maps with
the same irrational rotation number and the same number of (non-flat) critical
points, and let $h : S^1 \to S^1$ be a homeomorphism conjugating f to g, i.e., such that
$h \circ f = g \circ h$. If $h$ maps each critical point of $f$ to a corresponding critical point of
$g$, then $h$ is quasisymmetric.
\end{theorem}

In the special case of maps with a single critical point, this theorem was first proved by Yoccoz (unpublished, but see \cite[Corollary 4.6]{edsonwelington1}). Here, the presence of \emph{at least one} critical point is absolutely crucial: the statement is \emph{false} for diffeomorphisms. Indeed, there exist diffeomorphisms of the circle, even analytic ones, that are topologically
conjugate to an irrational rotation and yet no conjugacy between them is quasi-symmetric; see \cite[p.~75]{livrounidim}. 

Let us give a brief description of the main steps in the proof of Theorem \ref{qsrigidity}. 

\subsubsection*{The basic idea} The basic idea is to build for each multicritical circle map $f$ an associated fine grid in a \emph{canonical} way, and then apply Corollary \ref{corogrid}. By canonical here we mean that the partitions making up this fine grid must be defined in a dynamically invariant way, {\it i.e.,} in purely combinatorial terms. An obvious first attempt is to use the dynamical partitions $\mathcal{P}_n(c,f)$, where $c$ is a critical point of $f$, all of whose vertices lie in the forward orbit of $c$. But even if we skip levels (to circumvent the fact that $\mathcal{P}_{n+1}(c,f)$ is not a strict refinement of $ \mathcal{P}_n(c,f)$) and look at a subsequence of this sequence of partitions, we are in trouble because, whenever a partial quotient $a_{n+1}$ is very large, there are atoms of $\mathcal{P}_{n+1}(c,f)$ which are much smaller than the atoms of $\mathcal{P}_n(c,f)$ in which they are contained. 
We need to group some of these small atoms together, but to do that we first need to understand their geometry.

\subsubsection*{Bridges and critical spots}

Let $\Delta_0=f^{q_n}(I_{n+1}(c))\subset I_n(c)$, and consider its images under $f^{q_{n+1}}$, namely 
$\Delta_i=f^{iq_{n+1}}(\Delta_0)$, for $0\leq i\leq a_{n+1}-1$. These intervals constitute an exact partition of $I_n(c)\setminus I_{n+2}(c)$, and they are 
\emph{fundamental domains} of the first return map $f^{q_{n+1}}: I_n(c)\to f^{q_{n+1}}(I_n(c))$. Some of these fundamental domains may contain a number of critical points of the return map (at most equal to the total number of critical points of $f$). These fundamental domains are called \emph{critical spots}.
We also think of the interval $I_{n+2}(c)\subset I_n(c)$ as a critical spot.
The intervals between two consecutive critical spots are called \emph{bridges}. 
A fundamental lemma proved in \cite[Lemmas 4.2 and 4.3]{EdF} as a consequence of the real bounds states that the lengths of critical spots and (non-empty) bridges are always comparable with the length of $I_n(c)$. The cross-ratio inequality implies that the same comparability is true for the images of these intervals under $f^j$ with respect to the corresponding long atoms $f^j(I_n(c))$, $0\leq j\leq q_{n+1}-1$, (we also refer to these images as critical spots and bridges). Hence for each $n$ we can define a partition $\mathcal{P}_n^*(c,f)$ whose atoms are the short atoms of $\mathcal{P}_n(c,f)$ together with all bridges and critical spots that make up the long atoms of $\mathcal{P}_n(c,f)$. This new partition still has the property that any two of its atoms that are adjacent are comparable. Moreover, one easily sees that 
$\mathcal{P}_n^*(c,f)$ refines $\mathcal{P}_n(c,f)$ and is refined by $\mathcal{P}_{n+1}(c,f)$. This fact tells us that neither the sequence $\{\mathcal{P}_n^*(c,f)\}_{n\geq 0}$ nor any of its subsequences can be a fine grid. Hence we are not done yet. We need to understand the geometry of bridges, and here is where things get more interesting.

\subsubsection*{Almost parabolic maps} 
Each bridge is made up of a union of a number $\ell$ of fundamental domains for $f^{q_{n+1}}$. If $\ell$ is, say, less than $1,000$, we say that the bridge is \emph{short}; otherwise we say that the bridge is \emph{long}. 
Whenever $a_{n+1}$ is large, some bridges will certainly be long, and the restriction of $f^{q_{n+1}}$ to each one of these bridges will be very nearly a parabolic map at the center of a saddle-node bifurcation (see Figure \ref{newsaddle}). Such \emph{almost parabolic maps} can be described abstractly as follows 
(see \cite[p.~354]{edsonwelington1} or \cite[Def.~4.1]{EdF}). 

\begin{figure}[t]
\begin{center}~
\hbox to \hsize{
\psfrag{f}[][][1]{$\scriptstyle{f^{q_{n+1}}}$} 
\psfrag{F}[][][1]{$\;f^{q_{n+1}}$}
\psfrag{d}[][][1]{$\Delta'$}
\psfrag{D}[][][1]{$\Delta''$}
\psfrag{B}[][][1]{$B$}
\includegraphics[width=5.2in]{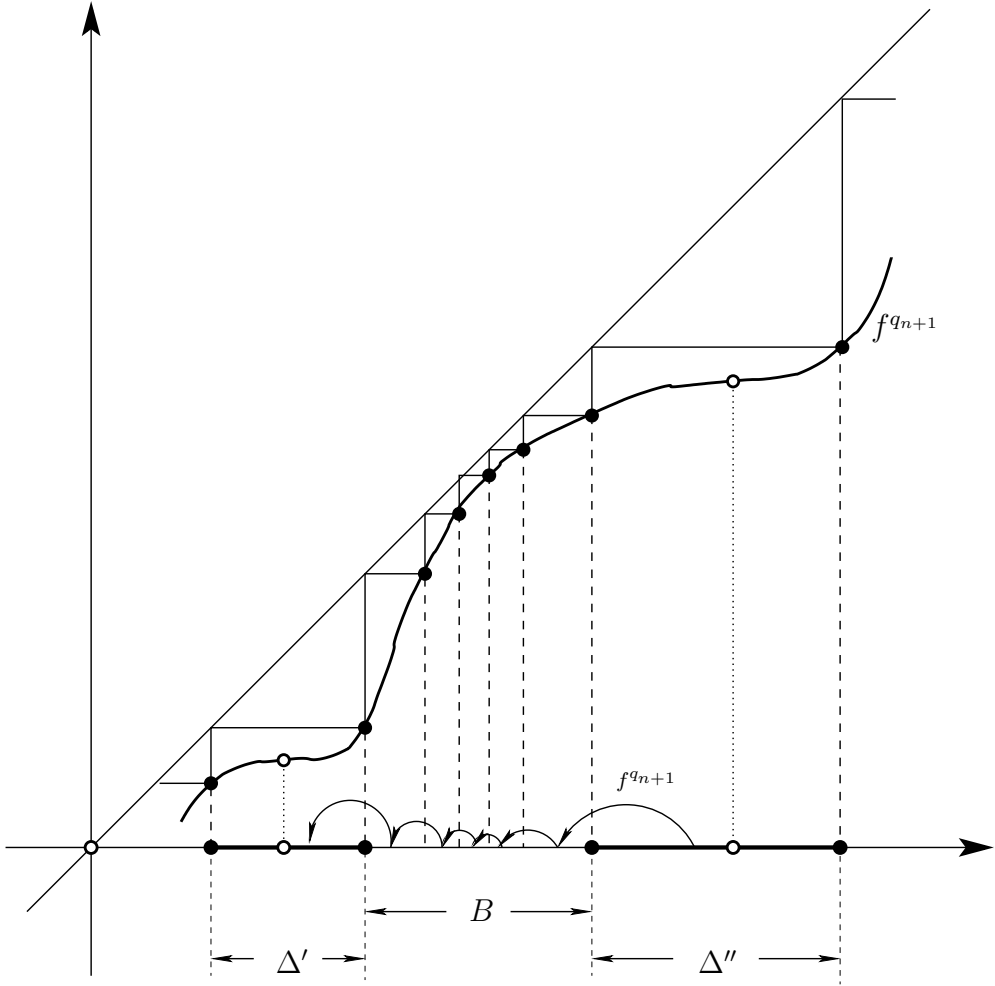}
   }
\end{center}
\caption[newsaddle]{\label{newsaddle} A bridge $B$ between two critical spots, $\Delta'$ and $\Delta''$.}
\end{figure}

\begin{definition}\label{almostparabolic}
 An \emph{almost parabolic map} is a $C^3$ diffeomorphism 
\[
 \phi:\, J_1\cup J_2\cup \cdots \cup J_\ell \;\to\; J_2\cup J_3\cup \cdots \cup J_{\ell+1} \ ,
\]
where $J_1,J_2, \ldots, J_{\ell+1}$ are consecutive intervals on the circle (or on the line), with the 
following properties. 
\begin{enumerate}
 \item[(i)] One has $\phi(J_k)= J_{k+1}$ for all $1\leq k\leq \ell$;
 \item[(ii)] The Schwarzian derivative of $\phi$ is everywhere negative.
\end{enumerate}
The positive integer $\ell$ is called the \emph{length} of $\phi$, and the positive real number 
\[
 \sigma =\min\left\{\frac{|J_1|}{|\cup_{k=1}^\ell J_k|}\,,\, \frac{|J_\ell|}{|\cup_{k=1}^\ell J_k|}     \right\}
\]
is called the \emph{width\/} of $\phi$.
\end{definition}

The basic geometric control of an almost parabolic map is provided by the following fundamental inequality due to Yoccoz, a complete proof of which can be found in \cite[App.~B, pp.~386-389]{edsonwelington1}.

\begin{lemma}[Yoccoz]\label{lemyoccoz}
 Let $\phi: \bigcup_{j=1}^\ell J_k \to \bigcup_{k=2}^{\ell+1} J_k$ be an almost parabolic map with length $\ell$ and 
width $\sigma$. There exists a constant $C_\sigma>1$ (depending on $\sigma$ but not on $\ell$) such that, 
for all $k=1,2,\ldots,\ell$, we have
\begin{equation}\label{yocineq}
 \frac{C_\sigma^{-1}|I|}{[\min\{k,\ell+1-k\}]^2} \;\leq\; |J_k| \;\leq\;  \frac{C_\sigma|I|}{[\min\{k,\ell+1-k\}]^2}\ ,
\end{equation}
where $I=\bigcup_{k=1}^\ell J_k$ is the domain of $\phi$. 
\end{lemma}

This lemma can be applied to $f^{q_{n+1}}|_{G}$ for every long bridge $G$, provided such restrictions have negative Schwarzian derivative. It turns out that the negative Schwarzian condition is always satisfied provided $n$ is sufficiently large -- this fact is proved in \cite[Lemma~4.6]{EdF}, see also \cite[Lemma 4.1]{EdFG} and \cite[Proposition A.7]{dFG2020}.

\subsubsection*{Balanced decomposition of bridges} 
Given an almost parabolic map $\phi: I\to \phi(I)$ as in Definition \ref{almostparabolic}, we perform a suitable  decomposition of its domain $I=\bigcup_{k=1}^{\ell} J_k$ as follows. Let $d\in \mathbb{N}$ be largest 
with the property that $2^{d+1}\leq \ell/2$, and consider the descending chain of (closed) intervals 
\[
 I=M_0\supset M_1\supset \cdots \supset M_{d+1}
\]
inductively defined as follows. Having defined $M_i$, the interval $M_{i+1}\subset M_i$ is chosen so that, if  $L_i,R_i$ denote the left and right connected components of $M_i\setminus M_{i+1}$, then both  $L_i$ and $R_i$ are the unions of exactly $2^i$ consecutive fundamental domains  
of $\phi$. Then Yoccoz's Lemma \ref{lemyoccoz} implies that $|L_i|\asymp |M_{i+1}|\asymp |R_i|$ for each $0\leq i\leq d$, with comparability constants 
depending only on the width $\sigma$ of $\phi$.
Note also that
\begin{equation}\label{balance1}
 I\;=\; \bigcup_{i=0}^{d}L_i\;\cup\;M_{d+1}\;\cup\; \bigcup_{i=0}^{d}R_i \ . 
\end{equation}
The intervals $M_i$ ($0\leq i\leq d+1$) are called \emph{central intervals}, whereas the intervals $L_i,R_i$ ($0\leq i\leq d$) are called \emph{lateral intervals} of the almost parabolic map $\phi$. Note that this construction is canonical in the sense that, whenever two almost parabolic maps are topologically conjugate (by a homeomorphism mapping fundamental domains to fundamental domains), the conjugacy will send the central (resp. lateral) intervals of one map to the central (resp. lateral) intervals of the other. The decomposition of the domain of $\phi$ that we have just decribed is called the \emph{balanced decomposition} of (the domain of) our almost parabolic map.

Now, for each \emph{long} bridge $G\in \mathcal{P}_n^*(c,f)$, consider the above balanced decomposition for $\phi=f^{q_{n+1}}|_{G}$. Here we are assuming that $n$ is large enough that $S\phi<0$. We define, for each such $n$, a new collection of intervals $\mathcal{P}_n^{**}(c,f)$ whose elements are: all short atoms of $\mathcal{P}_n(c,f)$, all \emph{short bridges} and critical spots of $\mathcal{P}_n^*(c,f)$, together with all the central intervals and all lateral intervals of every \emph{long bridge} of $\mathcal{P}_n^*(c,f)$. Note that this new collection $\mathcal{P}_n^{**}(c,f)$ is \emph{not} a partition of the circle, since, for example, the central intervals of each long bridge are \emph{nested}. 

\subsubsection*{A suitable fine grid} The desired fine grid $\{\mathcal{Q}_n(c,f)\}$ can now be constructed by a recursive procedure, which we state below. Informally, each partition $\mathcal{Q}_n(c,f)$ of the grid will be made up of those atoms in $\mathcal{P}_m^*(c,f)$ that are not long bridges together with elements of the balanced partitions of long bridges, which belong to $\mathcal{P}_m^{**}(c,f)$ for various levels $m\leq n$. 
The precise statement (taken from \cite[Prop.~5.2]{EdF}) is the following.

\begin{prop} \label{gridpprop}
There exists a fine grid $\{\mathcal{Q}_n\}$ in
$S^1$ with the following properties.
\begin{enumerate}
\item[($a$)] Every atom of $\mathcal{Q}_n$ is the union of at most $a=4N+3$ 
atoms of $\mathcal{Q}_{n+1}$, where $N$ is the number of critical points of $f$.
\item[($b$)] Every atom $\Delta\in \mathcal{Q}_n$ is a union of atoms of
$\mathcal{P}_m^*(c,f)$ for some $m\leq n$, and there are three possibilities:
\begin{enumerate}
\item[($b_1$)] $\Delta$ is a single atom of $\mathcal{P}_m^*(c,f)$;
\item[($b_2$)] $\Delta$ is a central interval of 
$\mathcal{P}_m^{**}(c,f)$; 
\item[($b_3$)] $\Delta$ is the union of at least two atoms of
$\mathcal{P}_{m+1}^*(c,f)$ contained in a single atom of $\mathcal{P}_m^{**}(c,f)$.
\end{enumerate}
\end{enumerate}
\end{prop}

The proof of Proposition \ref{gridpprop} is given in detail in \cite[pp.~5613-5614]{EdF} (see also a similar construction in \cite[Section 4]{EG}). The original  precursor of this result for \emph{unicritical} circle maps appears in \cite[Prop.~4.5]{edsonwelington1}. The situation there is a bit simpler, because there are no critical spots to be considered, hence no bridges in between.

\subsubsection*{Finishing the proof} The proof of Theorem \ref{qsrigidity} should now be clear. By hypothesis, the conjugacy $h$ sets a bijective correspondence between the critical points of $f$ and the critical points of $g$. Let $c$ be a critical point of $f$, and let $h(c)$ be the corresponding critical point of $g$. Then $h$ maps each partition $\mathcal{P}_n(c,f)$ onto the corresponding partition
$\mathcal{P}_n(h(c),g)$, sending critical spots to critical spots and bridges to bridges. Therefore if $\mathcal{G}_f=\{\mathcal{Q}_n(c,f)\}$ and $\mathcal{G}_g=\{\mathcal{Q}_n(h(c),g)\}$ are the fine grids for $f$ and $g$, respectively, given by Proposition \ref{gridpprop}, it follows that $h$ maps $\mathcal{G}_f$ bijectively onto $\mathcal{G}_g$. But then, by Corollary \ref{corogrid}, $h$ is quasisymmetric. This finishes the proof.

\subsection{Orbit rigidity and flexibility}\label{sec:orbflex}

As we have just pointed out, the rigidity phenomenon embodied by Corollary \ref{qsbdtype} is quite specific to the bounded combinatorics scenario. Indeed, the following result was recently obtained in \cite[Theorem C]{dFG2019}.

\begin{theorem}\label{ThmCdFG} There exists a set\, $\bm{R} \subset [0,1]$ of irrational numbers, which contains a residual set (in the Baire sense) and has full Lebesgue measure, for which the following holds.
For each $\rho\in\bm{R}$ there exist two $C^{\infty}$ multicritical circle maps $f, g: S^1\to S^1$ with the following properties:
\begin{enumerate}
 \item Both maps have the same rotation number $\rho$;
 \item The map $f$ has exactly one critical point $c_0$, whereas the map $g$ has exactly two critical points $c_1$ and $c_2$;
 \item The topological conjugacy between $f$ and $g$ that takes $c_0$ to $c_1$ is not quasisymmetric.
\end{enumerate}
\end{theorem}

Of course, by the above mentioned consequence of the real bounds, the set $\bm{R}$ only contains irrationals numbers of \emph{unbounded} type. Therefore, the geometric orbit structure is much richer in the unbounded case than in the bounded one. Another way to see this phenomenon is the following: let us say that two given orbits of a minimal circle homeomorphism $f$ are said to be \emph{geometrically equivalent} if there exists a quasisymmetric circle homeomorphism identifying both orbits and commuting with $f$. By the well-known theorem due to Herman and Yoccoz discussed in Section \ref{secAHY}, if $f$ is a smooth diffeomorphism with Diophantine rotation number, then any two orbits are geometrically equivalent. As already mentioned in this section (Corollary \ref{qsbdtype}), the same holds if $f$ is a critical circle map with rotation number of bounded type. By contrast, it was recently proved in \cite{dFG2019} that if $f$ is a critical circle map whose rotation number belongs to a certain full Lebesgue measure set in $(0,1)$, then the number of equivalence classes is \emph{uncountable}.

\section{Smooth rigidity}\label{smoothrigid}

The notion of \emph{smooth rigidity} first appeared in hyperbolic geometry in the sixties, through the seminal work of Mostow, who showed that the fundamental group (the \emph{topology}) of a complete, finite-volume hyperbolic manifold of dimension greater than two, completely determines its \emph{geometry}. In dynamical systems, smooth rigidity means that a finite number of dynamical invariants determines the fine scale structure of orbits. More precisely, maps that are topologically conjugate and share these invariants are in fact smoothly conjugate. Numerical observations \cite{feigetal, ostlundetal, sh} suggested in the early eighties that this was the case for $C^3$ critical circle maps with a single critical point and with irrational rotation number of bounded type. This was posed as a conjecture in several works by Lanford \cite{lanford1, lanford2}, Rand \cite{ostlundetal, rand1, rand2, rand3} and Shenker \cite{feigetal, sh} among others. We proceed to state the most recent results in this area, namely Theorem \ref{rigC4} and Theorem \ref{rigC3} below.

\begin{theorem}\label{rigC4} Let $f$ and $g$ be two $C^4$ circle homeomorphisms with the same irrational rotation number and with a unique critical point of the same odd type. Let $h$ be the unique topological conjugacy between $f$ and $g$ that maps the critical point of $f$ to the critical point of $g$. Then:
\begin{enumerate}
\item\label{Aitem1} $h$ is a $C^1$ diffeomorphism.
\item\label{Aitem2} $h$ is $C^{1+\alpha}$ at the critical point of $f$ for a universal $\alpha>0$.
\item\label{Aitem3} There exists a full Lebesgue measure set of rotation numbers (containing those of bounded type) for which the conjugacy $h$ is a global $C^{1+\alpha}$ diffeomorphism.
\end{enumerate}
\end{theorem}

By Theorem \ref{yoccoztheorem}, the rotation number is the unique invariant of the $C^0$ conjugacy classes of critical circle maps with no periodic orbits. Theorem \ref{rigC4} is saying that, inside each topological class, the order of the critical point is the unique invariant of the $C^1$ conjugacy classes. This is what we call \emph{rigidity}.

A delicate problem is to precisely determine ``how smooth" is the conjugacy $h$. By comparing with Section \ref{secAHY} we see that, on one hand, the presence of the critical point gives us more rigidity than in the case of diffeomorphisms: smooth conjugacy is obtained for all irrational rotation numbers, with no Diophantine conditions. On the other hand, examples have been constructed \cite{avila, edsonwelington1} showing that $h$ may not be globally $C^{1+\alpha}$ in general, even for real-analytic dynamics. It might be possible, but probably quite difficult, to obtain a sharp arithmetical condition on the rotation number that would allow us to decide whether the conjugacy is ``better than $C^1$" (see Section \ref{sec:conc} for further remarks).

\begin{theorem}\label{rigC3} Any two $C^3$ critical circle maps with a single critical point, with the same irrational rotation number of bounded type and with the same odd criticality are conjugate to each other by a $C^{1+\alpha}$ circle diffeomorphism, for some universal $\alpha>0$.
\end{theorem}

We remark that the statement of Theorem \ref{rigC3} was the precise statement of the \emph{rigidity conjecture} mentioned above. Together, Theorems \ref{rigC4} and \ref{rigC3} can be regarded as the state of the art concerning rigidity of critical circle maps with a single critical point. Theorem \ref{rigC4} was proved in \cite{GMdM2015} while Theorem \ref{rigC3} was proved in \cite{GdM2013}. Both papers build on earlier work of Herman, \'Swi\k{a}tek, the first named author, de Melo, Yampolsky, Khanin and Teplinsky among others \cite{hermanihes, swiatek, EdsonThesis, edsonETDS, edsonwelington1, edsonwelington2, yampolsky1, yampolsky2, yampolsky3, yampolsky4, khmelevyampolsky, khaninteplinsky}. While the proofs of both results are far beyond the scope of this survey, some ideas will be provided in sections \ref{subsecren} and \ref{seccomplex} below.

\medskip

What about dynamics with more critical points? Let $f$ be a $C^3$ multicritical circle map with irrational rotation number $\rho\in(0,1)$, unique invariant Borel probability measure $\mu$ and $N \geq 1$ critical points $c_i$, for $0\leq i\leq N-1$. As before, all critical points are assumed to be non-flat: in $C^3$ local coordinates around $c_i$, the map $f$ can be written as\, $t \mapsto t\,|t|^{d_i-1}$ for some $d_i>1$ (Definition \ref{defnaoflat}). Moreover, just as in Section \ref{seccritmaps} (recall Definition \ref{signature}), we define the \emph{signature} of $f$ to be the $(2N+2)$-tuple 
\[
(\rho\,;N;\,d_0,d_1,\ldots,d_{N-1};\,\delta_0,\delta_1,\ldots,\delta_{N-1}),
\]
where  $d_i$ is the criticality of the critical point $c_i$, and $\delta_i=\mu[c_i,c_{i+1})$ (with the convention that $c_{N}=c_0$).

Now consider two multicritical circle maps, say $f$ and $g$, with the same irrational rotation number. By Theorem \ref{yoccoztheorem}, they are topologically conjugate to each other. By elementary reasons, if $f$ and $g$ have the same signature there exists a circle homeomorphism $h$, which is a topological conjugacy between $f$ and $g$, identifying each critical point of $f$ with one of $g$ having the same criticality. As explained in Section \ref{secQS}, such conjugacy $h$ is a quasisymmetric homeomorphism (Theorem \ref{qsrigidity}).

\begin{question}\label{conjrigmulti} Is this conjugacy a smooth diffeomorphism?
\end{question}

Of course $h$ is the \emph{unique} conjugacy between $f$ and $g$ that can be smooth (in fact, as mentioned in Section \ref{sec:orbflex}, for Lebesgue almost every rotation number most conjugacies between $f$ and $g$ fail to be even quasisymmetric, see \cite{dFG2019} for precise statements). The following result follows by combining the recent papers \cite{EG,ESY,yampolsky5}.

\begin{theorem}\label{rigbicrit} Let $f$ and $g$ be real-analytic bi-critical circle maps with the same irrational rotation number, both critical points of cubic type and with the same signature. If their common rotation number is of bounded type, then the topological conjugacy $h$ is a $C^{1+\alpha}$ diffeomorphism. 
\end{theorem}

To the best of our knowledge, Theorem \ref{rigbicrit} is the first rigidity statement available for maps with more than one critical point. In other words, Question \ref{conjrigmulti} remains widely open.

\section{Renormalization Theory}\label{subsecren}

The main tool available to study rigidity problems in low dimensional dynamics is renormalization theory. Renormalization of a dynamical system with a marked point (usually a critical point) means a (suitably rescaled) first return map to a neighbourhood of such point. Thus, renormalization can be thought as a \emph{supra} dynamical system, acting on an infinite dimensional phase space (see Section \ref{secpairs} for precise definitions in the context of multicritical circle maps). In the context of one dimensional dynamics, the renormalization program was initiated by Dennis Sullivan in the eighties \cite{sullivanICM,sullivan}, and then carried out by mathematicians such as Yoccoz, Douady, Hubbard, Shishikura, McMullen, Lyubich, Martens, the first named author, de Melo, Yampolsky, van Strien and Avila among others.

A fundamental principle in this theory states that exponential convergence of renormalization orbits implies rigidity: topological conjugacies are actually smooth (when restricted to the attractors of the original systems). We refer the reader to \cite[Section VI.9]{livrounidim} for the seminal case of unimodal maps with bounded combinatorics (more precisely, see Theorem 9.4). For critical circle maps with a single critical point, this principle has been established by the first named author and de Melo for Lebesgue almost every irrational rotation number \cite[First Main Theorem]{edsonwelington1}, and extended later by Khanin and Teplinsky to cover all irrational numbers \cite[Theorem 2]{khaninteplinsky}. Adapting these previous approaches, this fundamental principle has been recently established for multicritical circle maps \cite[Theorem A]{EG}.

So the main goal in order to obtain rigidity, as in Section \ref{smoothrigid} above, is to establish geometric contraction of renormalization along critical circle maps with the same irrational rotation number and the same criticality. The dynamics of renormalization, however, is usually difficult to understand. To begin with, its phase space is neither bounded nor locally compact. Therefore, no recurrence is given a priori. This makes some basic dynamical questions, such as existence of attractors and periodic orbits, quite difficult to solve\footnote{As a first step, the real bounds (Theorem \ref{realbounds}) can be used to establish $C^r$ bounds for return maps, see for instance \cite[Appendix A]{edsonwelington1}. A standard Arzel\`a-Ascoli argument gives then pre-compactness of renormalization orbits.} (see Section \ref{sec:conc} for further comments). In particular, to prove exponential contraction is a challenging problem. In the case of a single critical point and real-analytic dynamics, exponential contraction was obtained in \cite{edsonwelington2} for rotation numbers of bounded type, and extended in \cite{khmelevyampolsky} to cover all irrational rotation numbers (see Theorem \ref{teoconvexpan} below). Both papers lean on complex dynamics techniques (that will be discussed in Section \ref{seccomplex} below), and therefore an additional hypothesis is required: the criticality at both critical points has to be an odd integer. These results have been recently extended in at least two directions: in \cite{GY} exponential contraction is obtained allowing non-integer criticalities which are close enough to an odd integer, while in \cite{GMdM2015} and \cite{GdM2013} exponential contraction is established for finitely smooth critical circle maps (still with odd integer criticalities, see Theorems \ref{geoconvC3BT} and \ref{geoconvC4} below). Finally, in the case of two critical points, it was recently proved in \cite{yampolsky5} both existence of periodic orbits and hyperbolicity (under renormalization) of those periodic orbits, for real-analytic bi-critical circle maps (with both critical points of cubic type). These results were later extended to bounded combinatorics in \cite{ESY}.

In the remainder of this section and the next one we sketch, with no details, the proofs of both theorems \ref{rigC4} and \ref{rigC3}.

\subsection{Renormalization of commuting pairs}\label{secpairs} As it was already clear in the early eighties \cite{feigetal,ostlundetal} it is convenient to construct a renormalization operator (see Definition \ref{renop}) acting not on the space of critical circle maps but on a suitable space of \emph{critical commuting pairs}, whose precise definition is the following.

\begin{definition}\label{critpair} A $C^r$ \emph{critical commuting pair} $\zeta=(\eta,\xi)$ consists of two $C^r$ orientation preserving homeomorphisms $\eta:I_{\eta}\to\eta(I_{\eta})$ and $\xi:I_{\xi}\to\xi(I_{\xi})$ where:
\begin{enumerate}
\item $I_{\eta}=[0,\xi(0)]$ and $I_{\xi}=[\eta(0),0]$ are compact intervals in the real line;
\item\label{comor} $\big(\eta\circ\xi\big)(0)=\big(\xi\circ\eta\big)(0) \neq 0$;
\item\label{la4} $D\eta(x)>0$ for all $x \in I_{\eta}\!\setminus\!\{0\}$ and $D\xi(x)>0$ for all $x \in I_{\xi}\!\setminus\!\{0\}$;
\item\label{pc} The origin is a non-flat critical point for both $\eta$ and $\xi$, with the same criticality.
\item\label{la5} Both $\eta$ and $\xi$ extend to $C^r$ homeomorphisms, defined on interval neighbourhoods of their respective domains, which commute around the origin.
\end{enumerate}
\end{definition}

\begin{figure}[t]
\begin{center}~
\hbox to \hsize{
\psfrag{0}[][][1]{$\scriptstyle{0}$} 
\psfrag{1}[][][1]{$\scriptstyle{\xi(0)}$}
\psfrag{-1}[][][1]{$\!\scriptstyle{\eta(0)}$}
\psfrag{eq}[][][1]{$\;\;\scriptstyle{\eta\circ\xi(0)=\xi\circ\eta(0)}$}
\psfrag{f0}[][][1]{$\xi$} 
\psfrag{f1}[][][1]{$\eta$}
\psfrag{I0}[][][1]{$I_\xi$}
\psfrag{I1}[][][1]{$I_\eta$}
\includegraphics[width=5.1in]{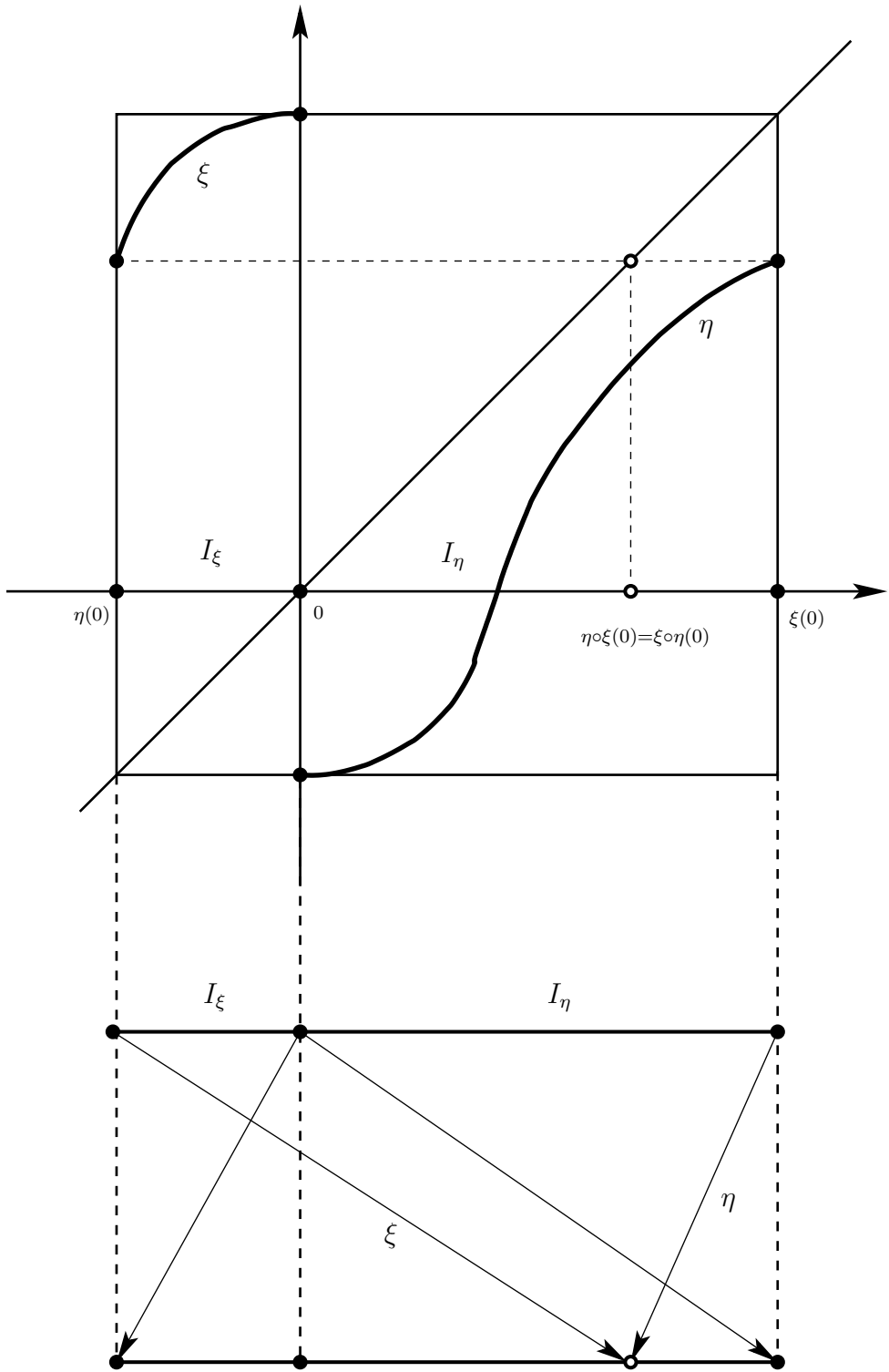}
   }
\end{center}
\caption[newpair]{\label{newpair} A critical commuting pair and its underlying interval exchange.}
\end{figure}

Any critical circle map $f$ with irrational rotation number $\rho$ induces a sequence of critical commuting pairs in a natural way: let $F$ be the lift of $f$ to the real line (for the canonical covering $t \mapsto e^{2\pi it}$) satisfying $DF(0)=0$ and $0<F(0)<1$. For each $n \geq 1$ let $\widehat{I}_n$ be the closed interval in the real line, adjacent to the origin, that projects under $t \mapsto e^{2\pi it}$ to $I_n$. Let $T:\R \to \R$ be the translation $x \mapsto x+1$, and define $\eta:\widehat{I}_n \to \R$ and $\xi:\widehat{I}_{n+1} \to \R$ as:$$\eta=T^{-p_{n+1}} \circ F^{q_{n+1}}\quad\mbox{and}\quad\xi=T^{-p_n} \circ F^{q_n}\,,$$where $\{p_n/q_n\}$ is the sequence of convergents defined in Section \ref{secdifeos}. It is not difficult to check that $(\eta|_{\widehat{I}_n}, \xi|_{\widehat{I}_{n+1}})$ is a critical commuting pair, usually denoted by $(f^{q_{n+1}}|_{I_n},f^{q_n}|_{I_{n+1}})$.

For a commuting pair $\zeta=(\eta,\xi)$ we denote by $\widetilde{\zeta}$ the pair $(\widetilde{\eta}|_{\widetilde{I_{\eta}}}, \widetilde{\xi}|_{\widetilde{I_{\xi}}})$, where tilde means \emph{rescaling} by the linear factor $1/|I_{\xi}|$. In other words, $|\widetilde{I_{\xi}}|=1$ and the length of $\widetilde{I_{\eta}}$ equals the ratio between those of $I_{\eta}$ and $I_{\xi}$.

Let $\zeta=(\eta,\xi)$ be a critical commuting pair according to Definition \ref{critpair}, and recall that $\big(\eta\circ\xi\big)(0)=\big(\xi\circ\eta\big)(0) \neq 0$. Let us suppose that $\big(\xi\circ\eta\big)(0) \in I_{\eta}$ (see Figure \ref{newpair}) and define the \emph{height} $\chi(\zeta)$ of $\zeta$ as $a\in\nt$ if$$\eta^{a+1}\big(\xi(0)\big) < 0 \leq \eta^a\big(\xi(0)\big)\,,$$and $\chi(\zeta)=\infty$ if no such $a$ exists. The height of the pair $(f^{q_{n+1}}|_{I_n},f^{q_n}|_{I_{n+1}})$ induced by a critical circle map $f$ is exactly $a_{n+1}$, where $\rho(f)=[a_0,a_1,...]$. For $\zeta=(\eta,\xi)$ with $\big(\xi\circ\eta\big)(0) \in I_{\eta}$ and $\chi(\zeta)=a<\infty$, the pair$$\big(\eta|_{[0,\eta^a(\xi(0))]}\,,\,\eta^a \circ \xi|_{I_\xi}\big)$$is again a commuting pair, and if $\zeta$ is induced by a critical circle map, \emph{i.e.},$$\zeta=\big(f^{q_{n+1}}|_{I_n}\,,\,f^{q_n}|_{I_{n+1}}\big)\,,$$then we have$$\big(\eta|_{[0,\eta^a(\xi(0))]}\,,\,\eta^a \circ \xi|_{I_\xi}\big)=\big(f^{q_{n+1}}|_{I_{n+2}}\,,\,f^{q_{n+2}}|_{I_{n+1}}\big)\,.$$This motivates the following definition.

\begin{definition}\label{renop} Let $\zeta=(\eta,\xi)$ be a critical commuting pair with $\big(\xi\circ\eta\big)(0) \in I_{\eta}$. We say that $\zeta$ is \emph{renormalizable} if $\chi(\zeta)=a<\infty$. In this case we define the \emph{renormalization} of $\zeta$ as$$\mathcal{R}(\zeta)=\left(\widetilde{\eta}|_{\widetilde{[0,\eta^a(\xi(0))]}}\,,\,\widetilde{\eta^a \circ \xi}|_{\widetilde{I_\xi}}\right).$$
\end{definition}

A critical commuting pair is a special case of a \emph{generalized interval exchange map} of two intervals, and the renormalization operator defined above is just the restriction of the \emph{Zorich accelerated version} of the \emph{Rauzy-Veech renormalization} for interval exchange maps (see for instance \cite{yoccoziem}). However we will keep in this survey the classical terminology for critical commuting pairs.

If $\chi(\mathcal{R}^j(\zeta))<\infty$ for $j \in \{0,1,...,n-1\}$ we say that $\zeta$ is \emph{n-times renormalizable}, and if $\chi(\mathcal{R}^j(\zeta))<\infty$ for all $j\in\nt$ we say that $\zeta$ is \emph{infinitely renormalizable}. The space of all infinitely renormalizable commuting pairs is the natural phase-space for renormalization. For such a pair, the irrational number whose continued fraction expansion equals$$\big[\chi\big(\zeta\big), \chi\big(\mathcal{R}(\zeta)\big),...,\chi\big(\mathcal{R}^n(\zeta)\big),\chi\big(\mathcal{R}^{n+1}(\zeta)\big),... \big]$$is, by definition, the \emph{rotation number} of the critical commuting pair $\zeta$ (note that if $\zeta$ is induced by a critical circle map with irrational rotation number, then it is infinitely renormalizable and both definitions of rotation number coincide).

For any real number $x$ denote by $\lfloor x\rfloor$ the \emph{integer part} of $x$, that is,
the greatest integer less than or equal to $x$. Also, denote by $\{x\}$ the \emph{fractional part} of $x$, that is, $\{x\}=x-\lfloor x\rfloor\in[0,1)$. Recall that the \emph{Gauss map} $G:[0,1]\to[0,1]$ is given by
\[
G(\rho)=\left\{\frac{1}{\rho}\right\}\mbox{ for $\rho\neq 0$\,, and $G(0)=0$.}
\]
Both $\Q\cap[0,1]$ and $[0,1]\setminus\Q$ are $G$-invariant. Under the action of $G$, all rational numbers in $[0,1]$ eventually land on the fixed point at the origin, while the irrationals remain in the union $\bigcup_{k \geq 1}\left(\frac{1}{k+1},\frac{1}{k}\right)$. Moreover, for any $\rho\in(0,1)\setminus\Q$ and any $j\in\nt$ we have that $G^j(\rho)\in\left(\frac{1}{k+1},\frac{1}{k}\right)$ if, and only if, $a_j=k$, where as before $a_j$ denotes the partial quotients of $\rho$. Indeed, if $\rho= [a_{0} , a_{1} , a_{2} , \cdots ]$ belongs to $\big(1/(k+1),1/k\big)$, then $1/\rho=a_0+[a_{1} , a_{2} , \cdots ]$ and then $a_0=\left\lfloor\frac{1}{\rho}\right\rfloor=k$ and $G(\rho)=[a_{1} , a_{2} , \cdots ]$. In particular, the Gauss map acts as a \emph{left shift} on the continued fraction expansion of $\rho$, and therefore the action of the renormalization operator on the rotation number is given by
\begin{equation}\label{rengauss}
\rho\big(\mathcal{R}(\zeta)\big)=G\big(\rho(\zeta)\big)=\sigma\big([a_0,a_1,\dots]\big)=[a_1,\dots]\,.
\end{equation}In particular, the renormalization operator $\mathcal{R}$ maps topological classes of (infinitely renormalizable) critical commuting pairs into topological classes.

\subsection{Exponential convergence of renormalizations}\label{secexpconv} For critical circle maps with a single critical point of some odd integer criticality, the renormalization operator $\mathcal{R}$ exhibits an infinite-dimensional \emph{horseshoe-like} attractor. More precisely, there exists a pre-compact $\mathcal{R}$-invariant set $\Lambda$, consisting of real-analytic critical commuting pairs with irrational rotation number, such that the action of $\mathcal{R}|_{\Lambda}$ is topologically conjugate to the two-sided shift $\sigma$ acting on $(\nt\cup\infty)^{\Z}$ (the action being taken over the continued fraction expansion of the rotation number, as in \eqref{rengauss} above). Moreover, any given real-analytic pair with irrational rotation number converges to $\Lambda$ geometrically. More precisely, the following holds.

\begin{theorem}\label{teoconvexpan} There exists a universal constant $\lambda$ in $(0,1)$ with the following property: given two real-analytic unicritical commuting pairs $\zeta_1$ and $\zeta_2$ with the same irrational rotation number and the same odd type at the critical point, there exists a constant $C>0$ such that$$d_{r}\big(\mathcal{R}^n(\zeta_1),\mathcal{R}^n(\zeta_2)\big) \leq C\lambda^n$$for all $n\in\nt$ and for any $0 \leq r < \infty$, where $d_r$ denotes the $C^r$ metric.
\end{theorem}

This theorem was obtained by the first named author and de Melo \cite{edsonwelington2} for rotation numbers of bounded type, and extended by Khmelev and Yampolsky \cite{khmelevyampolsky} to cover all irrational rotation numbers. As already mentioned, its proof relies on holomorphic methods, to be discussed in Section \ref{seccomplex}.

The link between real-analytic critical circle maps and finitely smooth ones is given by the following result, first obtained in \cite[Theorem D]{GdM2013} for the $C^0$ metric, and then extended in \cite[Theorem 11.1]{GMdM2015} for the $C^{r-1}$ metric.

\begin{theorem}[Shadowing]\label{compacto} There exists a $C^{\omega}$-compact set $\mathcal{K}$ of real-analytic critical commuting pairs with the following property: for any $r \geq 3$ there exists a constant $\lambda=\lambda(r) \in (0,1)$ such that given a $C^r$ critical circle map $f$ with irrational rotation number there exist $C>0$ and a sequence $\{f_n\}_{n\in\nt}$ contained in $\mathcal{K}$ such that$$d_{r-1}\big(\mathcal{R}^n(f),f_n\big)\leq C\lambda^n\quad\mbox{for all $n\in\nt$,}$$and such that the pair $f_n$ has the same rotation number as the pair $\mathcal{R}^n(f)$ for all $n\in\nt$. Here $d_{r-1}$ denotes the $C^{r-1}$ distance in the space of $C^{r-1}$ critical commuting pairs.
\end{theorem}

An important tool used in \cite{GdM2013} in order to construct the compact set $\mathcal{K}$ is the following notion. Let $I$ be a compact interval in the real line, and let $f:I\to\R$ be a $C^r$ map. Then there exists a $C^r$ map $F:\Omega\to\C$, where $\Omega$ is a neighbourhood of $I$ in the complex plane, with $F|_I=f$, $\frac{\partial}{\partial\bar{z}}F(z)=0$ when $\mathrm{Im}\,z=0$ and moreover$$\frac{\frac{\partial}{\partial\bar{z}}F(z)}{(\mathrm{Im}\,z)^{r-1}} \to 0$$uniformly as $\mathrm{Im}\,z$ goes to zero (see \cite[Lemma 2.1, p. 623]{grasandsswia} for one possible way to construct such extension $F$). In this case, we say that $F$ is \emph{asymptotically holomorphic} of order $r$ in $I$. To prove Theorem \ref{compacto}, one first extends the initial critical commuting pair given by $f$ to a pair of asymptotically holomorphic maps defined in an open complex neighbourhood of each original interval. To the best of our knowledge, asymptotically holomorphic maps were first used in one-dimensional dynamics by Lyubich in the early nineties \cite{lyubich}, and later by Graczyk, Sands and \'Swi\k{a}tek in \cite{grasandsswia}. One of its fundamental properties, an \emph{almost Schwarz inclusion}, controls the action of asymptotically holomorphic maps on some complex domains symmetric to the real line (behaving, at small scale, almost as isometries of a suitable hyperbolic metric, see \cite[Proposition 2, p. 629]{grasandsswia} and also Lemma \ref{lemmaschwdFdM2} below). When combined with Theorem \ref{realbounds} (the real bounds), such control provides geometric bounds on the quasiconformal distortion of the renormalizations of the previously mentioned extensions (one does not study the dynamics of these extensions, just their geometric behaviour). With these bounds at hand, the deformations from $\mathcal{R}^n(f)$ to $f_n$ (in order to prove Theorem \ref{compacto}) are based on the following corollary of the classical \emph{Measurable Riemann Mapping Theorem} (also known as Ahlfors-Bers theorem, see \cite[Section V.C]{ahlfors}), whose proof can be found in \cite[Proposition 5.5]{GdM2013}.

\begin{prop} For any bounded domain $U$ in the complex plane there exists a number $C(U)>0$, with $C(U) \leq C(W)$ if $U \subseteq W$, such that the following holds: let $\big\{G_n:U \to G_n(U)\big\}_{n\in\nt}$ be a sequence of quasiconformal homeomorphisms\footnote{Let $\Omega\subset\mathbb{C}$ be a domain and let $K \geq 1$. An orientation-preserving homeomorphism $f:\Omega \to f(\Omega)$ is \emph{$K$-quasiconformal} if it is absolutely continuous on lines and satisfies$$\left|\overline{\partial}f(z)\right|\leq\left(\frac{K-1}{K+1}\right)\big|\partial f(z)\big|\quad\mbox{for Lebesgue a.e. $z\in\Omega$\,.}$$The \emph{Beltrami coefficient} of such homeomorphism $f$ is the measurable function $\mu_f:\Omega\to\mathbb{D}$ given by$$\mu_f(z)=\frac{\overline{\partial}f(z)}{\partial f(z)}\quad\mbox{for Lebesgue a.e. $z\in\Omega$\,.}$$} such that
\begin{itemize}
\item The domains $G_n(U)$ are uniformly bounded: there exists $R>0$ such that $G_n(U) \subset B(0,R)$ for all $n\in\nt$.
\item $\mu_n \to 0$ in the unit ball of $L^{\infty}$, where $\mu_n$ is the Beltrami coefficient of $G_n$ in $U$.
\end{itemize}
Then given any domain $V$ such that $\overline{V} \subset U$ there exist $n_0\in\nt$ and a sequence $\big\{H_n:V \to H_n(V)\big\}_{n \geq n_0}$ of biholomorphisms such that:$$\|H_n-G_n\|_{C^0(V)} \leq C(U)\left(\frac{R}{d\big(\partial V,\partial U\big)}\right)\|\mu_n\|_{\infty}\quad\mbox{for all $n \geq n_0$,}$$where $d\big(\partial V,\partial U\big)$ denote the Euclidean distance between the boundaries of $U$ and $V$.
\end{prop}

This Proposition is applicable to many other situations -- see for example \cite[Section 5.5]{CT2020}.

Those were some of the main ideas in order to prove Theorem \ref{compacto}. With theorems \ref{teoconvexpan} and \ref{compacto} at hand, the following two results, which are \cite[Theorem C]{GdM2013} and \cite[Theorem B]{GMdM2015} respectively, can be obtained.

\begin{theorem}\label{geoconvC3BT} There exists a universal constant $\lambda \in (0,1)$ such that given two $C^3$ critical circle maps $f$ and $g$ with the same irrational rotation number of bounded type and the same odd integer criticality, there exists $C=C(f,g)>0$ such that for all $n\in\nt$ we have$$d_0\big(\mathcal{R}^n(f),\mathcal{R}^n(g)\big) \leq C\lambda^n\,.$$
\end{theorem}

\begin{theorem}\label{geoconvC4} There exists a universal constant $\lambda \in (0,1)$ such that given two $C^4$ critical circle maps $f$ and $g$ with the same irrational rotation number and the same odd integer criticality, there exists $C=C(f,g)>0$ such that for all $n\in\nt$ we have$$d_2\big(\mathcal{R}^n(f),\mathcal{R}^n(g)\big) \leq C\lambda^n\,.$$
\end{theorem}

By the already mentioned \cite[First Main Theorem]{edsonwelington1} and \cite[Theorem 2]{khaninteplinsky}, theorems \ref{geoconvC3BT} and \ref{geoconvC4} imply theorems \ref{rigC3} and \ref{rigC4} respectively. We remark that to obtain theorems \ref{geoconvC3BT} and \ref{geoconvC4} from theorems \ref{teoconvexpan} and \ref{compacto} some hard analysis is involved, concerning a Lipschitz estimate for the renormalization operator (when restricted to suitable bounded pieces of topological conjugacy classes). We refer the reader to \cite[Lemma 4.1]{GMdM2015} for more details.

Following the terminology from hyperbolic dynamics, the rotation number of a critical circle map can be thought as an \emph{unstable} coordinate for the renormalization operator (recall from \eqref{rengauss} that $\mathcal{R}$ acts as the Gauss map on the rotation number), while the criticality of its single critical point (being preserved under renormalization) is a \emph{central} one. By Theorem \ref{teoconvexpan}, \emph{stable} manifolds for renormalization are obtained by fixing these two coordinates (see \cite[Section 2.3]{yampolsky4} for a nice description of this hyperbolic picture).

For multicritical circle maps with $N \geq 2$ critical points, a richer picture emerges: the criticalities $d_i$ are again central directions, whereas the parameters $\delta_i$ (the distance between consecutive critical points view under the unique invariant measure, see Definition \ref{signature}) are \emph{new unstable directions} (such expanding behaviour comes from the associated rescaling process explained in Section \ref{secpairs}). This suggests that, inside topological classes of multicritical commuting pairs, one should be able to see many different horseshoe-like attractors for renormalization, each one containing infinitely many periodic orbits. This hyperbolic picture has not been established yet, which is why Question \ref{conjrigmulti} remains open in its general form.

\section{Holomorphic methods}\label{seccomplex} 

In this section we will survey some of the complex-analytic ideas that play a decisive role in the theory of (multi)critical circle maps. Since these ideas are quite deep, the narrative to follow is by necessity very sketchy. For the general theory of complex dynamics we refer the reader to the books \cite{CGbook,dFdMbook,mclivro1,Milnorbook}.

The use of holomorphic methods in the study of renormalization and rigidity of one-dimensional dynamical systems was started by Sullivan in the mid-eighties (see \cite{sullivan}). Since the theory for circle maps follows in parallel with the corresponding theory for unimodal maps, and borrows substantially from it, we need to talk a bit about the latter first. 

\subsection{Sullivan's program} We have already mentioned the general ansatz relating renormalization convergence and rigidity. If we are given two topologically conjugate one-dimensional maps $f$ and $g$ which are infinitely renormalizable (say with some restrictions on their combinatorics), and if we know that the $C^0$ distances between their successive renormalizations contract to zero \emph{at an exponential rate}, then the conjugacy between $f$ and $g$ should actually be smooth. 
Hence the goal becomes to establish exponential contraction of renormalizations. 
The strategy laid down by Sullivan in \cite{sullivan} (and explained in greater detail in \cite[Ch.~VI]{livrounidim})
to achieve this goal can be roughly described as follows.
\begin{enumerate}
\item First get geometric bounds on the orbits of the critical points of the (real) one-dimensional systems.  These so-called \emph{real a-priori bounds} should be robust enough that, even if we start with maps which have only a mild, finite degree of smoothness, their successive renormalizations will converge $C^0$ exponentially fast to the subspace consisting of real-analytic maps.  
\item Use such real a-priori bounds to show that the topological conjugacy between
the two systems has slightly more geometric regularity than being merely continuous: it is actually \emph{quasi-symmetric} (at least when restricited to the post-critical sets of both systems).
\item Complexify the given real dynamical systems (when they are real-analytic), in other words, find suitable complex-analytic extensions of these systems. 
\item Using the real bounds in
(1) and the mild geometric control in (2), get \emph{complex a priori bounds}
for the complexified systems. These bounds are usually bounds on the moduli of
certain annuli (typically fundamental domains for the complexified systems). Such bounds yield a strong form of compactness.
\item Extend the renormalization operator to the complexified dynamical systems. This operator will, in a suitable domain, be a compact operator due to step (4). 
\item Use the bounds and compactness in (3) and some suitable infinite-dimensional
version of Schwarz’s lemma to establish the desired contraction property
of the underlying renormalization operator.
\end{enumerate}

In the context of (real-analytic) unimodal maps of the interval, Sullivan realized that the relevant complex-analytic dynamical systems are \emph{quadratic-like maps} (or more generally polynomial-like maps), and was therefore able to use the theory developed by Douady and Hubbard \cite{DH} for such maps. Recall that a quadratic-like map is a proper, degree two holomorphic branched covering map $F:U\to V$ between two topological disks $U,V\subset \mathbb{C}$ with $U$ compactly contained in $V$, branched at a unique critical point $c\in U$. 
The \emph{modulus} of $F$ is by definition the conformal modulus of the annulus $V\setminus U$. The set $\mathcal{K}_F=\bigcap_{n\geq 0} F^{-n}(V)$ is called the \emph{filled-in Julia set} of $F$. It is a totally invariant set under the dynamics, and it is compact due to the fact that $F$ is proper. Every point in $U\setminus\mathcal{K}_F$ has a finite orbit that eventually lands in the outer annulus $V\setminus U$. This annulus threfore works as a \emph{fundamental domain} for the dynamics outside the filled-in Julia set. A central fact about quadratic-like maps is the \emph{straightening theorem} of Douady and Hubbard \cite{DH}: every quadratic like map is quasiconformally conjugate to an actual quadratic polynomial map.

A quadratic-like map $F:U\to V$ is said to be \emph{renormalizable} if one can find a sub-disk $D\subset U$ compactly contained in $U$ and containing $c$ and an integer $p\geq 2$ such that $F^p|_{D}: D\to F^p(D)\subset V$ is well-defined, and again a quadratic-like map. This new map, with $p$ smallest possible and suitably rescaled (via a complex affine map), is called the \emph{first renormalization} of $F$, and denoted $\mathcal{R}F$.
The number $p$ is called the \emph{renormalization period} of $F$, denoted $p(F)$. 
If all successive renormalizations $\mathcal{R}^2F=
\mathcal{R}(\mathcal{R}F),\ldots, \mathcal{R}^{n}F=\mathcal{R}(\mathcal{R}^{n-1}F), \ldots$  are well-defined, then we say that $F$ is \emph{infinitely renormalizable}. If in addition all periods $p_n=p(\mathcal{R}^nF)$ form a bounded sequence, we say that $F$ infinitely renormalizable of \emph{bounded type}. The \emph{complex bounds} proved by Sullivan guarantee that if one starts with a real-analytic, infinitely renormalizable quadratic unimodal map $f$ of bounded type on the real line, then after a finite number $N$ of iterations, the renormalized unimodal  maps $\mathcal{R}^nf$ will be restrictions of quadratic-like maps $F_n$ with $F_{n+1}=\mathcal{R}F_n$ for all $n\geq N$, and moreover the moduli $\mathrm{mod}(\mathcal{R}^nF)$ ($n\geq N$) will be bounded from below. In particular, the sequence $(\mathcal{R}^{n}F)_{n\geq N}$ will be a pre-compact family (in the topology of uniform convergence on compacta), and every limit of such renormalization sequence will be a quadratic-like map. Here and throughout, all holomorphic maps considered commute with complex conjugation, \emph{i.e.,} are symmetric about the real axis. 

The crucial feature of quadratic-like maps in this theory, very closely related to the straightening theorem, is that they are amenable to what Sullivan calls a \emph{pull-back} argument. If $F_i:U_i\to V_i$, $i=0,1$, are two symmetric, topologically conjugate quadratic-like maps, and if $h$ is a quasisymmetric homeomorphism of the real line which sends the post-critical set of $F_0$ to the post-critical set of $F_1$, then $F_0$ and $F_1$ are quasiconformally conjugate. More precisely, there exists a quasiconformal homeomorphism $H:V_0\to V_1$ such that $H\circ F_0=F_1\circ H$; in addition, the quasiconformal dilatation of $H$ depends only on the conformal moduli $\mathrm{mod}(V_i\setminus U_i)$ ($i=0,1$) and on the quasisymmetric distortion of $h$.

The existence of such a conjugacy already allows us to speak of the \emph{quasiconformal} or \emph{Teichm\"uller} distance between $F_0$ and $F_1$, defined as 
\begin{equation}\label{juliateich}
 d_T(F_0,F_1) \;=\; \inf_{\phi} \log{\frac{1+\|\mu_\phi\|_{\infty}}{1-\|\mu_\phi\|_{\infty}}}\ ,
\end{equation}
the infimum being taken over all quasi-conformal conjugacies $\phi$ between $F_0$ and $F_1$. 
This is in fact a pseudo-distance: its value will be zero whenever the two maps are \emph{conformally} conjugate. It turns out that the Julia set of an (symmetric) infinitely renormalizable quadratic-like map carries no \emph{invariant line fields} (equivalently, no non-zero invariant Beltrami differentials). This is another consequence of the straightening theorem. Thus, for every quasiconformal conjugacy $\phi$ as above we have that $\mu_\phi$ vanishes a.e. on the (filled-in) Julia set of $F_0$. In particular, when calculating $\|\mu_\phi\|_\infty$ in the right-hand side of \eqref{juliateich}, we only need to look at the values of $\mu_\phi(z)$ for $z\in V_0$. 

It is immediate from the definition that the Teichm\"uller distance is weakly contracted under renormalization: any conjugacy between $F_0$ and $F_1$ restricts to a conjugacy between $\mathcal{R}(F_0)$ and 
$\mathcal{R}(F_1)$.

Now, let $H$ be a quasiconformal conjugacy between $F_0$ and $F_1$, say the one constructed via the pull-back argument. Its Beltrami differential $\mu_H=\overline{\partial} H/\partial H$ is invariant under $F_0$, and therefore it can be used to generate a \emph{path} of (pairwise qc-conjugate) quadratic-like maps joining $F_0$ to $F_1$. To see this, define $\mu_t=t\mu_H$ for all $t\in \mathbb{C}$ such that $|t|< \|\mu_H\|_{\infty}^{-1}$ then integrate each $\mu_t$ using the measurable Riemann mapping theorem to  get a (normalized) quasiconformal homeomorphism $H_t$, and then define $F_t=H_t\circ F_0\circ H_t^{-1}$. Such a path is called a \emph{Beltrami path} joining $F_0$ to $F_1$. 

As one can see from the definitions given so far, renormalization maps Beltrami paths to Beltrami paths. Some Beltrami paths are more efficient than others, in the sense that they are close to being ``geodesics'' in the Teichm\"uller metric. It will usually be the case that a very efficient Beltrami path joining $F_0$ to $F_1$ will be mapped to an inefficient Beltrami path joining $\mathcal{R}(F_0)$ to $\mathcal{R}(F_1)$: the image path ``coils''. It turns out that one can put this coiling property into more quantitative terms, and the result is a form of Schwarz's lemma in infinite dimensions.{\footnote{However, we warn the reader that the renormalization ``operator'' is \emph{not} a complex-analytic operator.}}

There are some difficulties with carrying out the details of this approach. One is the fact that the domain and range of a quadratic-like map vary with the map itself, so it is hard to set up the renormalization procedure as an actual operator on a space of maps defined over a fixed domain. Another difficulty is the fact that, if we are given two quadratic-like maps and they both restrict to the same quadratic unimodal map on the line, then they should be regarded as essentially the same dynamical system; however, their Teichm\"uller distance, according to the definition given above, will not be zero! Sullivan soon realized that a way to circumvent these difficulties is to take an inverse limit of the dynamics off the filled-in Julia set. To wit, if $F:U\to V$ is the given quadratic-like map, one considers the inverse system
\[
 \cdots \to F^{-(n+1)}(V\setminus \mathcal{K}_F)\to F^{-n}(V\setminus \mathcal{K}_F)\to \cdots F^{-1}(V\setminus \mathcal{K}_F)\to V\setminus \mathcal{K}_F\ ,
\]
where each map, being a restriction of $F$, is an unbranched $2$-to-$1$ holomorphic covering. 
The inverse limit of this system, denoted $\mathcal{L}(F)$, is a \emph{Riemann surface lamination} in a natural way. This object is locally homeomorphic to the product of a disk by a Cantor set, and the chart transitions are holomorphic on the leaves. The construction is  canonical in the sense that, if $F$ varies (but stays in the same topological conjugacy class), then topologically $\mathcal{L}(F)$ does not change at all. Only its conformal structure changes.
Moreover, a quasiconformal conjugacy between two such maps induces a homeomorphism between the two corresponding laminations which is quasiconformal on each leaf. Hence, one can speak of the (\emph{moduli space} or) \emph{Teichm\"uller space} of such lamination. It then follows that renormalization induces an operator on such Teichm\"uller space.

Using these ideas, Sullivan was able to carry out the strategy outlined in steps (1)-(6) above almost completely in the bounded-type case. 
We say ``almost'' because in step (6) he was forced to settle for something less than exponential contraction. Sullivan made an ingenious use of the theory of Riemann surface laminations, and used the Teichm\"{u}ller theory of such objects (which he largely developed on the fly) to prove a (non-uniform) version of Schwarz's lemma in this context, which in turn allowed him to prove renormalization convergence without a rate. 
The exponential convergence of renormalizations for bounded type infinitely renormalizable maps was finally achieved by McMullen \cite{mclivro1,mclivro2} by a different route, using his theory of \emph{rigidity of towers}. 

\begin{remark}
 The theory of Riemann surface laminations is a beautiful subject in its own right. See \cite{ghys} for a nice exposition. 
\end{remark}

\subsection{Holomorphic commuting pairs}
In his PhD thesis \cite{EdsonThesis}, the first named author took up the task of carrying out as much as possible of Sullivan's program in the context of critical circle maps with a single critical point of cubic type. Steps (1) and (2) of Sullivan's strategy were already in place due to the works of Herman and Swiatek (Theorem \ref{realbounds}) and Yoccoz (Theorem \ref{qsrigidity} in the unicritical case). 

The key to the remaining steps is an analogue of the quadratic-like maps of Douady and Hubbard, a holomorphic dynamical system that somehow extends the real commuting pairs arising as successive renormalizations of a critical circle map. This is the central contribution of \cite{EdsonThesis} and of the subsequent paper \cite{edsonETDS}. Here are the relevant  definitions, taken almost verbatim from \cite[p.~346]{edsonwelington2}.

\begin{definition}
 By a {\it bowtie} we mean $4$-tuple $({\mathcal O}_\xi,{\mathcal O}_\eta,{\mathcal O}_\nu, {\mathcal V})$ of simply-connected domains
 in the complex
plane such that:
\begin{enumerate}
\item[($a$)] Each ${\mathcal O}_\gamma$ is a Jordan domain whose
closure is contained in ${\mathcal V}$; 
\item[($b$)] We have $\overline{{\mathcal O}}_\xi\cap \overline{{\mathcal
O}}_\eta= \{0\}\subseteq {\mathcal O}_\nu$;
\item[($c$)] The sets ${\mathcal O}_\xi\setminus {\mathcal O}_\nu$, ${\mathcal
O}_\eta\setminus {\mathcal O}_\nu$, ${\mathcal O}_\nu\setminus {\mathcal O}_\xi$
and ${\mathcal O}_\nu\setminus {\mathcal O}_\eta$ are non-empty and connected.
\end{enumerate}
\end{definition}

\begin{figure}[t]
\begin{center}~
\hbox to \hsize{
\psfrag{0}[][][1]{$\scriptstyle{\!\!0}$} 
\psfrag{a}[][][1]{$\scriptstyle{a}$}
\psfrag{b}[][][1]{$\scriptstyle{b}$}
\psfrag{n0}[][][1]{$\scriptstyle{\!\eta(0)}$}
\psfrag{e0}[][][1]{$\scriptstyle{\xi(0)}$}
\psfrag{n}[][][1]{$\eta$}
\psfrag{nu}[][][1]{$\nu$}
\psfrag{e}[][][1]{$\xi$}
\psfrag{O1}[][][1]{$\mathcal{O}_\xi$}
\psfrag{O2}[][][1]{$\mathcal{O}_\eta$}
\psfrag{O3}[][][1]{$\mathcal{O}_\nu$}
\psfrag{V}[][][1]{$\mathcal{V}$}
\psfrag{r}[][][1]{$\scriptstyle{\mathrm{Re}(z)}$}
\includegraphics[width=5.6in]{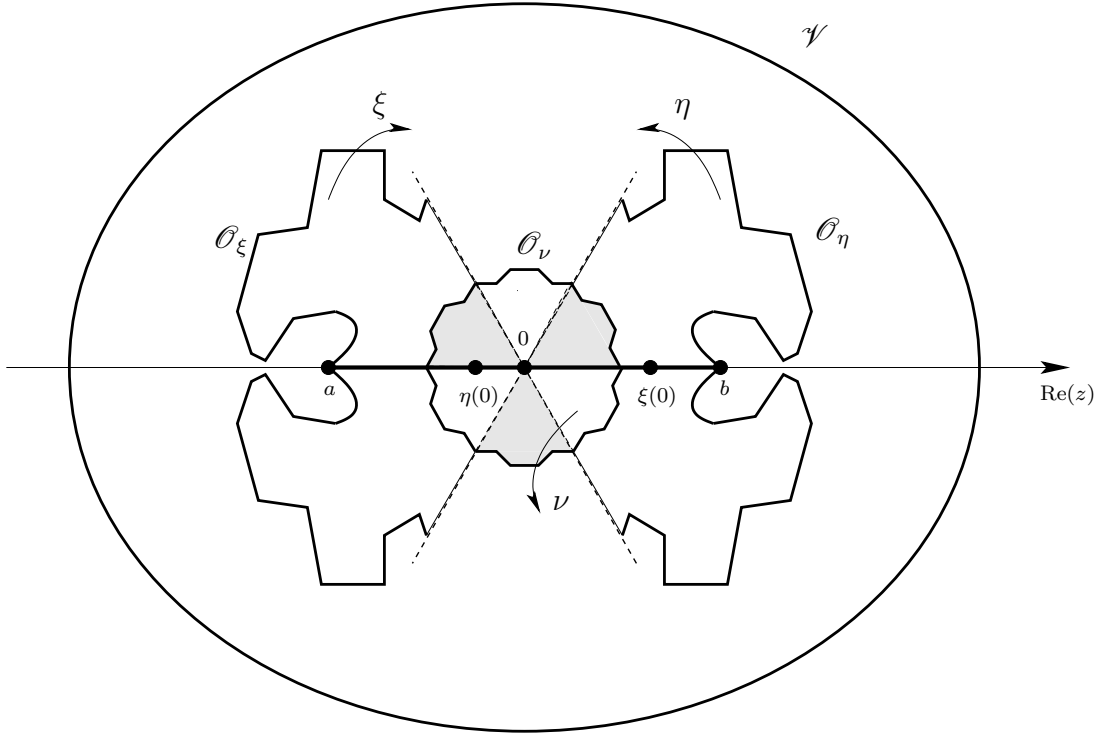}
   }
\end{center}
\caption[holopair]{\label{holopair} A holomorphic commuting pair.}
\end{figure}

\begin{definition}
Let $({\mathcal O}_\xi,{\mathcal O}_\eta,{\mathcal O}_\nu, {\mathcal V})$ be a bowtie. 
 A holomorphic pair $\Gamma$ with domain ${\mathcal U}={\mathcal
O}_\xi\cup {\mathcal O}_\eta\cup {\mathcal O}_\nu$ and co-domain ${\mathcal V}$ is the
dynamical system generated by three holomorphic maps $\xi:{\mathcal O}_\xi\to
{\mathbb C}$, $\eta:{\mathcal O}_\eta\to {\mathbb C}$ and $\nu:{\mathcal O}_\nu\to {\mathbb C}$
satisfying the following conditions (see Figure \ref{holopair}).
\begin{enumerate}
\item[H$_1$] Both $\xi$ and $\eta$ are univalent onto
${\mathcal V}\cap {\mathbb C}(\xi(J_\xi))$ and ${\mathcal V}\cap {\mathbb
C}(\eta(J_\eta))$ respectively, where $J_\xi={\mathcal O}_\xi\cap {\mathbb R}$
and $J_\eta={\mathcal O}_\eta\cap {\mathbb R}$. (Notation: ${\mathbb C}(I)=({\mathbb
C}\setminus {\mathbb R})\cup I$.)

\item[H$_2$] The map $\nu$ is a $3$-fold branched cover
onto ${\mathcal V}\cap {\mathbb C}(\nu(J_\nu))$, where $J_\nu={\mathcal
O}_\nu\cap {\mathbb R}$, with a unique critical point at $0$.

\item[H$_3$] We have ${\mathcal O}_\xi\ni \eta(0)<0<\xi(0)\in
{\mathcal O}_\eta$, and the restrictions $\xi|[\eta(0),0]$ and
$\eta|[0,\xi(0)]$ constitute a critical commuting pair.

\item[H$_4$] Both $\xi$ and $\eta$ extend holomorphically
to a neighborhood of zero, and we have
$\xi\circ\eta(z)=\eta\circ\xi(z)=\nu(z)$ for all $z$ in that
neighborhood. 

\item[H$_5$] There exists an integer $m\ge 1$, called the
{\it height} of $\Gamma$, such that $\xi^m(a)=\eta(0)$, where $a$ is
the left endpoint of $J_\xi$; moreover, $\eta(b)=\xi(0)$, where $b$ is
the right endpoint of $J_\eta$.
\end{enumerate}

\end{definition}

The relevant dynamical system here, which we will still denote by $\Gamma$, is the pseudo-semigroup generated by the three maps $\xi,\eta, \nu$.  
The interval $J=[a,b]$ is called the {\it long dynamical interval} of
$\Gamma$, whereas $\Delta=[\eta(0),\xi(0)]$ is the {\it short dynamical
interval} of $\Gamma$. They are both forward invariant under the
dynamics, as the reader can easily check. The {\it rotation number} of $\Gamma$ is by definition the
rotation number of the critical commuting pair of $\Gamma$ obtained by restriction to the real line (condition
H$_3$). We say that the holomorphic commuting pair $\Gamma$ has \emph{geometric boundaries} if $\partial\mathcal{U}$ and $\partial\mathcal{V}$ are \emph{quasi-circles}{\footnote{A quasi-circle, we recall, is the image of a round disk under a quasiconformal homeomorphism of the plane.}}.

\begin{remark}
Examples of holomorphic commuting pairs with arbitrary rotation number and arbitrary heights can be constructed directly from the Arnold family. This is carefully done in \cite[\S 4]{edsonETDS}. We should also point out that there is nothing special about cubic critical points. Holomorphic commuting pairs can be defined so as to have a critical point with any odd-power criticality whatever. To see how this is done, the reader should consult Arlane Vieira's thesis \cite{arlane} (see also \cite{yampolsky6}).
\end{remark} 

It turns out that the 
holomorphic pair $\Gamma$ can be renormalized:  the first
renormalization of the critical commuting pair of $\Gamma$ extends in
a natural way to a holomorphic pair ${\mathcal R}(\Gamma)$ with the same
co-domain ${\mathcal V}$. See Prop.~2.3 in \cite{edsonETDS} for the detailed construction of ${\mathcal R}(\Gamma)$. Renormalization is defined in such a way that the restriction of the renormalized holomorphic pair $\mathcal{R}(\Gamma)$ to the real line is the critical commuting pair that represents the renormalization of the critical commuting pair $(\xi|_{[\eta(0),0]}\,,\,\eta|_{[0,\xi(0)]})
$. 

\subsection{Pull-back argument}
The first main result in \cite{EdsonThesis} (or \cite{edsonETDS}) is the following analogue of Sullivan's \emph{pull-back argument}.

\begin{theorem}[Pull-back Argument]\label{pullbackthm}
 Let $\Gamma$ and $\Gamma'$ be holomorphic
pairs with geometric boundaries and let $h:J\to J'$ be a
quasisymmetric conjugacy between the restrictions of $\Gamma$ and $\Gamma'$ to their respective long dynamical intervals $J$ and $J'$. Then there exists
a quasiconformal conjugacy $H:{\mathcal V}\to {\mathcal V}'$ between $\Gamma$ 
and $\Gamma'$ which is an extension of $h$.
\end{theorem}

The proof is more involved than that of the original pull-back argument, for the following reason. In the quadratic-like case, we know by the straightening theorem of Douady-Hubbard that every quadratic-like map is quasiconformally conjugate to a quadratic polynomial, and the latter does not have wandering domains (due to Sullivan's no-wandering-domains theorem \cite{sullivan0}). Hence quadratic-like maps do not have wandering domains. By contrast, holomorphic pairs could in principle have wandering domains. To deal with their putative existence, one needs to use a form of quasiconformal surgery (something called the \emph{qc-sewing lemma} of  L.~Bers, see \cite[Lem.~3.2]{edsonETDS}). Wandering domains are only ruled out \emph{a posteriori}, combining Theorem \ref{pullbackthm} with the fact that holomorphic pairs constructed from the Arnold family do not carry such domains (see \cite[Th.~4.2]{edsonETDS}).  

\subsection{Complex bounds}
Another important fact about holomorphic commuting pairs is that the class of such objects contains all limits of successive renormalizations of a critical circle map (or critical commuting pair). Moreover, we have \emph{complex bounds} for renormalization, in the following sense.

\begin{theorem}[Complex Bounds]\label{complexbounds}
 Let $f: S^1\to S^1$ be
a real-analytic critical circle map with arbitrary irrational rotation
number. Then there exists $n_0=n_0(f)$ such that for all $n\ge n_0$ the 
$n$-th renormalization of $f$ extends to a holomorphic pair with geometric
boundaries whose fundamental annulus has conformal
modulus bounded from below by a universal constant.  
\end{theorem}

We think of the unit circle $S^1=\mathbb{R}/\mathbb{Z}$ as embedded in the infinite cylinder $\mathbb{C}/\mathbb{Z}$, and we use on latter the conformal metric induced from the standard Euclidean metric $|dz|$ of the complex plane via the exponential map $\exp(z)= e^{2\pi iz}$. Note that $\mathrm{Im}\,z$ is well-defined for every $z\in \mathbb{C}/\mathbb{Z}$ (it is simply the imaginary part of any one of its pre-images under the exponential). Note also that every real-analytic circle map $f$ as above has a holomorphic extension to an annular neighborhood of the unit circle inside $\mathbb{C}/\mathbb{Z}$. 

The main step in the proof of Theorem \ref{complexbounds} is to establish a geometric estimate showing that, for all 
sufficiently large $n$, the appropriate inverse branch of $f^{q_{n+1}}$
maps a sufficiently large disk around the $n$-th renormalization
domain $I_n\cup I_{n+1}$ well within itself. Here, ``sufficiently large'' means large with respect to the size of $I_n\cup I_{n+1}$. 
For each $m\geq 1$, let $D_m\subset \mathbb{C}/\mathbb{Z}$ denote the disk having as one of its diameters the interval $[f^{q_{m+1}}(c), f^{q_m-q_{m+1}}(c)]\subset S^1$ containing the critical point $c${\footnote{It is easy to see that $[f^{q_{m+1}}(c), f^{q_m-q_{m+1}}(c)]\supset I_m\cup I_{m+1}$. }}. 
The geometric estimate is the following (once again, the statement is taken almost verbatim from \cite[Prop.~3.2]{edsonwelington2}).

\begin{prop}\label{trapprop}
There exist universal constants $B_0$ and
$B_1$ and for each $N\ge 1$ there exists $n(N)$  such that
for all $n\ge n(N)$ the inverse branch $f^{-q_{n+1}+1}$ taking
$f^{q_{n+1}}(I_n)$ back to $f(I_n)$ is a well-defined univalent map over
$\Omega_{n,N}=(D_{n-N}\setminus S^1)\cup f^{q_{n+1}}(I_n)$, and for
all $z\in \Omega_{n,N}$ we have   
\[
\frac{\mathrm{dist}\,\left(f^{-q_{n+1}+1}(z), f(I_n)\right)}{|f(I_n)|}
\;\le\; B_0\left(\frac{\mathrm{dist}\,(z,I_n)}{|I_n|}\right)+B_1
\ .
\]
\end{prop}

As stated, Theorem \ref{complexbounds} was proved in \cite[\S 3]{edsonwelington2}. But the story behind it is a bit more involved. The first version of the complex bounds in the present context was proved in \cite{EdsonThesis} (also \cite{edsonETDS}) under two further assumptions on $f$, namely 
\begin{enumerate}
\item[(i)] the rotation number of $f$ is of bounded type; 
\item[(ii)] $f$ is an \emph{Epstein map}. 
\end{enumerate}
We say that a real analytic circle map is \emph{Epstein} if its lift to the real line has a holomorphic extension $F$ to a neighborhood of the real axis in the complex plane in such a way that $F$ has inverse branches which are globally defined in the upper (or lower) half-plane. The main examples of Epstein circle maps are the maps in the Arnold family introduced earlier (see \S \ref{secArnold}). The proof presented in \cite{EdsonThesis,edsonETDS} makes use of the so-called \emph{sector theorem} of Sullivan (see \cite{sullivan}; the version used in the circle case is in fact the one proved in \cite{edson1998}). However, the sector theorem can only be used under the bounded type assumption (i). 

That assumption was removed by Yampolsky in \cite{yampolsky1}, using a special case of Proposition \ref{trapprop}.  Assuming that the map $f$ is Epstein, he exploits in full the idea of \emph{Poincar\'e neighborhood trapping}, which we briefly explain. Let $J\subseteq \mathbb{R}$ be a bounded open interval, and write $\mathbb{C}(J)=\mathbb{C}\setminus (\mathbb{R}\setminus J)$. If $\phi: \mathbb{C}(J)\to \mathbb{C}(\phi(J))$ is a symmetric holomorphic map, then $\phi$ maps each Poincar\'e neighborhood $\mathbb{P}_\theta(J)=\{z:\,\mathrm{angle}(z,J)\geq \theta\}$ into a corresponding Poincar\'e neighborhood $\mathbb{P}_\theta(\phi(J))$ with the same angle $\theta$. Here, $0<\theta<\pi$ and $\mathrm{angle}(z,J)$ denotes the angle at $z$ under which $z$ views the interval $J$. This simple but fundamental fact is easily seen to be a consequence of Schwarz's lemma. 

The Poincar\'e neighborhood trapping idea used in Yampolsky's approach works because he is assuming that $f$ is Epstein. But if we abandon the  latter hypothesis, then this tool is no longer directly applicable. In order to prove Theorem \ref{complexbounds}, one needs the following ``relaxed'' version of Poincar\'e neighborhood trapping  (whose statement is taken verbatim from \cite[Lem.~3.3]{edsonwelington2}).  

\begin{lemma}\label{lemmaschwdFdM2} For every small $a>0$, there exists
$\theta(a)>0$ satisfying $\theta(a)\to 0$ and $a/\theta(a)\to 0$ as
$a\to 0$, such that the following holds. Let $F:\mathbb{D}\to \mathbb{C}$
be univalent and symmetric about the real axis, and assume $F(0)=0$,
$F(a)=a$. Then for all $\theta\geq \theta(a)$ we have 
$F\left(\mathbb{P}_\theta ((0,a))\right)\subseteq
\mathbb{P}_{(1-a^{1+\delta})\theta}((0,a))$, where $0<\delta<1$ is an absolute constant. 
\end{lemma}

This lemma is applicable to other situations -- see for example \cite{CvST2017}. It is a precursor to the more general \emph{almost Schwarz inclusion} lemma for asymptotically holomorphic maps already mentioned in Section \ref{secexpconv}, see \cite{grasandsswia}. 

\subsection{McMullen's dynamic inflexibility theorem} Let $f$ and $g$ be two real-analytic critical circle maps and let $h$ be a quasisymmetric conjugacy between $f$ and $g$, mapping the critical point $c_f$ of $f$ to the critical point $c_g$ of $g$. Suppose $h$ is $C^{1+ \epsilon}$ at the critical point $c_f$, for some $\epsilon>0$. Then, it is not difficult to prove (using the real bounds) that the $C^0$ distance between $\mathcal{R}^nf$ and $\mathcal{R}^ng$ converges to zero exponentially fast as $n\to \infty$ (and this, as we have already seen, leads to $C^1$ rigidity). Now, one way to guarantee that $h$ is $C^{1+\epsilon}$ at the critical point is if we know that $h$ extends to a quasiconformal homeomorphism $H$ (conjugating, say, the holomorphic extensions of $f$ and $g$ on a small neighborhood of their critical points) which happens to be \emph{$C^{1+\alpha}$-conformal} at the critical point $c_f$, in the following sense.

\begin{definition}
 We say that a map $\phi:\widehat{\mathbb{C}}\to \widehat{\mathbb{C}}$ is {\it $C^{1+\alpha}$-conformal} at $p\in\widehat{\mathbb{C}}$ (for
some $\alpha>0$) if the complex derivative $\phi'(p)$ exists and we have 
\[
\phi(z)=\phi(p)+\phi'(p)(z-p)+O(|z-p|^{1+\alpha})
\]
for all $z$ near $p$.
\end{definition}

In \cite{mclivro2}, McMullen developed a powerful theory that yields in particular a criterion for a conjugacy between two holomorphic dynamical systems to be $C^{1+\alpha}$-conformal at a point. His definition of holomorphic dynamical system is very broad, encompassing rational or trancendental maps, Kleinian groups, etc, as well as all possible geometric limits of such systems. 

In order to state McMullen's criterion, we need some preparatory definitions. Our exposition here is borrowed from \cite[\S 7]{edsonwelington2}. 

Let us denote by $\mathcal{V}(\widehat{\mathbb{C}}\times \widehat{\mathbb{C}})$ the
set of all analytic hypersurfaces of $\widehat{\mathbb{C}}\times \widehat{\mathbb{
C}}$. We topologize $\mathcal{V}(\widehat{\mathbb{C}}\times \widehat{\mathbb{C}})$ as follows. If $F\subseteq
\widehat{\mathbb{C}}\times \widehat{\mathbb{C}}$ is a hypersurface, its boundary
$\partial F={\overline{F}}\setminus F$ is closed in $\widehat{\mathbb{C}}\times
\widehat{\mathbb{C}}$. Hence, given $F\in \mathcal{V}(\widehat{\mathbb{C}}\times
\widehat{\mathbb{C}})$ and a sequence $F_i\in \mathcal{V}(\widehat{\mathbb{C}}\times
\widehat{\mathbb{C}})$, declare $F_i\to F$ if 

\begin{enumerate}
\item[($a$)] $\partial F_i\to \partial F$ in the Hausdorff metric on 
closed subsets of $\widehat{\mathbb{C}}\times \widehat{\mathbb{C}}$;

\item[($b$)] For each open set $U\subseteq \widehat{\mathbb{C}}\times
\widehat{\mathbb{C}}$ there exist $f,f_i:U\to \mathbb{C}$ such that $U\cap F=f^{-1}(0)$, $U\cap F_i=f_i^{-1}(0)$, each
$f,f_i$ vanishes to order one on $F,F_i$ respectively, and the sequence
$f_i$ converges uniformly to $f$ on compact subsets of $U$.
\end{enumerate}

Define a set to be closed in $\mathcal{V}(\widehat{\mathbb{C}}\times\widehat{\mathbb{C}})$  if it contains the limits of all its convergent sequences. As McMullen shows in \cite[Ch.~9]{mclivro2}, the space
$\mathcal{V}(\widehat{\mathbb{C}}\times \widehat{\mathbb{C}})$ with this topology is separable and metrizable.  

\begin{definition}
A \emph{holomorphic dynamical system} is a subset $\mathcal{F}\subseteq
\mathcal{V}(\widehat{\mathbb{C}}\times \widehat{\mathbb{C}})$. The elements of $\mathcal{F}$ are its \emph{holomorphic relations}. 
\end{definition}

One is primarily interested in \emph{closed} holomorphic dynamical systems, in other words, those which are closed subsets of $\mathcal{V}(\widehat{\mathbb{C}}\times \widehat{\mathbb{C}})$. 
The {\it geometric topology} on the space of all closed holomorphic dynamical
systems is by definition the Hausdorff topology on the space of closed
subsets of $\mathcal{V}(\widehat{\mathbb{C}}\times \widehat{\mathbb{C}})$. As McMullen shows in \cite[Ch.~9]{mclivro2}, the
geometric topology is typically non-Hausdorff (hence non-metrizable), but it is always sequentially compact. 

We also need the following notions introduced by McMullen.

\begin{enumerate}

\item \emph{Deep point} Given a compact set  $\Lambda\subseteq \mathbb{C}$ and a positive number $\delta$, we say that
a point $p\in \Lambda$ is a {\it $\delta$-deep point} 
of $\Lambda$ if for every $r>0$ the largest
disk contained in $D(p,r)$ which does not intersect $\Lambda$ has
radius $\le r^{1+\delta}$.

\item \emph{Saturation}
Given a holomorphic dynamical system $\mathcal{F}$, we define its saturation $\mathcal{F}^{\text{sat}}$ to be the closure in $\mathcal{V}(\widehat{\mathbb{C}}\times\widehat{\mathbb{C}})$ of the set whose elements are the intersections $F\cap U$, where $F\in \mathcal{F}$  and $U\subseteq \widehat{\mathbb{C}}\times
\widehat{\mathbb{C}}$ is open. 

\item \emph{Non-linearity}
A holomorphic dynamical system $\mathcal{F}\subseteq \mathcal{V}(\widehat{\mathbb{C}}\times \widehat{\mathbb{C}})$ is said to be non-linear if it does not leave
invariant a parabolic line field in $\widehat{\mathbb{C}}$.

\item \emph{Twisting}
A (closed) holomorphic dynamical system $\mathcal{F}\subseteq \mathcal{V}(\widehat{\mathbb{C}}\times \widehat{\mathbb{C}})$ is said to be twisting if
every holomorphic dynamical system quasi-conformally conjugate to $\mathcal{F}$
is non-linear.

\item \emph{Uniform twisting}
A family $\{\mathcal{F}_\alpha\}$ of holomorphic dynamical systems is said to be
uniformly twisting if every geometric limit of the family of saturations
$\{\mathcal{F}_\alpha^{\text{sat}}\}$ is a twisting dynamical system.

\item \emph{The family $(\mathcal{F},\Lambda)$}
Given $\mathcal{F}\subseteq \mathcal{V}(\widehat{\mathbb{C}}\times \widehat{\mathbb{C}})$ and a compact set $\Lambda$ in the Riemann sphere, we define a family
$(\mathcal{F},\Lambda)$ of holomorphic dynamical systems in the following way. For each baseframe $\omega$ in the convex-hull $\text{ch}(\Lambda)$ of
$\Lambda$ in hyperbolic 3-space, let $T_\omega$ be the fractional linear
transformation that sends $\omega$ onto the standard baseframe $\omega_0$ at
$(0,1)\in \mathbb{C}\times \mathbb{R}_+\equiv \mathbb{H}^3$. Define $(\mathcal{F},\omega)$ to be the dynamical system $T_\omega^*(\mathcal{F})$, the pull-back
of $\mathcal{F}$ by $T_\omega$. Then let $(\mathcal{F},\Lambda)$ be the family of
all $(\mathcal{F},\omega)$ as $\omega$ ranges through the baseframes in
$\text{ch}(\Lambda)$. 
\end{enumerate}

Now we have everything we need to state McMullen's dynamic inflexibility theorem. The proof is given in  \cite[p.~166]{mclivro2}.  

\begin{theorem}[Dynamic Inflexibility]\label{inflexthm}
Let $\mathcal{F}\subseteq\mathcal{V}(\widehat{\mathbb{C}}\times\widehat{\mathbb{C}})$ be a holomorphic dynamical system and let $\Lambda\subseteq \widehat{\mathbb{C}}$ be a compact set. If $(\mathcal{F},\Lambda)$ is uniformly twisting and $\phi: \widehat{\mathbb{C}}\to
\widehat{\mathbb{C}}$ is a K-quasiconformal conjugacy between $\mathcal{F}$ and
another holomorphic dynamical system $\mathcal{F}'$, then for each $\delta$-deep
point $p\in \Lambda$ the map $\phi$ is $C^{1+\alpha}$ conformal at $p$, for
some $\alpha>0$. \qed
\end{theorem}

\subsection{Proof of exponential convergence of renormalizations} With McMullen's dynamic inflexiblity theorem at hand, one can prove Theorem \ref{teoconvexpan}. By the complex bounds, every sufficiently high renormalization of a real-analytic critical commuting pair extends to a holomorphic commuting pair with good geometric control. Moreover, a quasisymmetric conjugacy between two such renormalized critical commuting pairs (mapping critical point to critical point) extends to a quasiconformal conjugacy between the corresponding renormalized holomorphic commuting pairs, by the pull-back argument. All one has to do, then, is to prove two things:  (a) that the critical point of a holomorphic commuting pair is $\delta$-deep for some $\delta>0$; and (b) that the full holomorphic dynamical system generated by a holomorphic commuting pair is uniformly twisting in its limit set.  The precise statements -- modulo the notion of \emph{good geometric control}, which we do not define here -- are as follows.

\begin{theorem}[Deep Critical Point] \label{deeppoints}
Let $\Gamma$ be a holomorphic pair with arbitrary rotation number and limit set $\mathcal{K}_\Gamma$. Then there exists
$\delta>0$ such that the critical point of $\Gamma$ is a $\delta$-deep point
of $\mathcal{K}_\Gamma$. 
\end{theorem}

\begin{theorem}[Small Limit Sets Everywhere] \label{smalllimitsets}
Let $\Gamma$ be a
holomorphic pair with good geometric control and arbitrary irrational rotation number, and let $\mathcal{K}_\Gamma$ be its limit set. Then for each $z_0\in \mathcal{K}_\Gamma$ and each $r>0$ there exists a pointed
domain $(U,y)$ with $|z_0-y|\asymp r$ and $\mathrm{diam}(U)\asymp r$, and there exist some iterate of $\Gamma$ mapping $(U,y)$ onto a pointed domain
$(V,0)$ univalently with bounded distortion. In particular, $U$
contains a conformal copy of some renormalization of $\Gamma$ whose limit set has size commensurable with $r$. 
\end{theorem}

These results are exact analogues of results obtained by McMullen in the context of (bounded-type, infinitely renormalizable) quadratic-like maps. Used in combination with Theorem \ref{inflexthm}, they yield the exponential convergence of renormalizations of Theorem~\ref{teoconvexpan} in the bounded type case. Theorem \ref{deeppoints} was proved in \cite{edsonwelington2} as stated here, without any assumption on the rotation number (other than being irrational). In that same paper, Theorem \ref{smalllimitsets} is stated and proved under the assumption that the rotation number is an irrational of bounded combinatorial type. This assumption 
was removed by Khmelev and Yampolsky in \cite{khmelevyampolsky}. When the sequence of partial quotients of the continued-fraction development of the rotation number is \emph{unbounded}, renormalization orbits may accumulate on commuting pairs having a fixed point (being in particular non-renormalizable). Such fixed point is necessarily \emph{parabolic} (with multiplier one), since the limiting pair is accumulated by pairs with no fixed points. Roughly speaking, the idea developed by Khmelev and Yampolsky was to apply the theory of \emph{parabolic bifurcations} (see Douady \cite{Do2}, Shishikura \cite{shishi1,shishi2} and references therein) to holomorphic commuting pairs, in order to understand the geometry of the domain of definition of pairs with arbitrarily small rotation number. With this at hand, the authors were able in the end to adapt, to the unbounded type case, the proof of Theorem \ref{teoconvexpan}  for the bounded type case explained above, see \cite[sections 6 and 7]{khmelevyampolsky}.

\subsection{Hyperbolicity of renormalization}\label{sechypren} In the previous subsection we have finally established Theorem \ref{teoconvexpan}, which assures exponential convergence of renormalization of \emph{real-analytic} critical commuting pairs with the same irrational rotation number and the same odd type at the critical point. As explained in Section \ref{secexpconv}, this dynamical picture can be promoted to critical commuting pairs with a \emph{finite degree of smoothness}, as in Theorem \ref{geoconvC3BT} and Theorem \ref{geoconvC4}. These two results can be regarded as the state of the art concerning exponential convergence of renormalization of critical circle maps (with a single critical point). As explained in Section \ref{smoothrigid}, they imply the rigidity theorems \ref{rigC4} and \ref{rigC3}, which are the main goal of this survey.

At this point, one would like to discuss the \emph{hyperbolicity} of renormalization (in the sense of Smale, \emph{i.e.}, uniform contraction/expansion on the tangent bundle). To give a meaning to this problem, one first needs to endow the phase-space of the renormalization operator with a smooth structure (a Banach manifold structure) on which $\mathcal{R}$ is (Fr\'echet) differentiable. As it turns out, this is a difficult problem that obstructs the hyperbolicity discussion directly in the space of critical commuting pairs. To overcome this problem, at least for real-analytic pairs, a crucial idea in this area was developed by Yampolsky in \cite{yampolsky3,yampolsky4}. Roughly speaking, his idea was to replace the renormalization operator $\mathcal{R}$, acting on the space of commuting pairs, with an analytic operator, the \emph{cylinder} renormalization operator, defined on a complex-analytic Banach manifold. This operator was constructed in \cite[Section 7]{yampolsky3}, while hyperbolicity of periodic orbits and the construction of the corresponding stable manifolds were given in \cite[sections 8 and 9]{yampolsky3}. Finally, hyperbolicity of the whole horseshoe-like attractor for the cylinder renormalization operator was obtained in \cite{yampolsky4} and re-obtained in \cite[Section 8]{khmelevyampolsky}.

\section{Concluding remarks}\label{sec:conc}

We end this survey with some remarks, conjectures and open questions on multicritical circle maps. 

\medskip

\begin{enumerate}

\item Recall that in Section \ref{seccritmaps} we introduced the notion of \emph{signature} of a multicritical circle map (see Def.~\ref{signature}). We may re-state Question \ref{conjrigmulti} as follows. Let $f,g:S^1\to S^1$ be two $C^3$ multicritical circle maps with the same signature, and let $h:S^1\to S^1$ be a conjugacy between $f$ and $g$ such that $h$ maps each critical point of $f$ to a corresponding critical point of $g$. Is $h$ a $C^1$ diffeomorphism? Are there conditions on the rotation number that make $h$ better than $C^1$? To the best of our knowledge, no rigidity results are available for maps with $N \geq 3$ critical points. As mentioned in Section \ref{secBlaschke} (see Remark \ref{remzak}), a construction similar to the one developed by Zakeri, in order to prove Theorem \ref{TeoZak}, should be useful as a starting point.

\vspace{.3cm}

\item What about (multi)critical circle maps with \emph{non-integer} criticalities? Not even the existence of periodic orbits (for renormalization) in the unicritical case has been established yet (but see \cite{GY}). For unimodal maps this problem has been solved by Martens in \cite{marco98}, but it is not clear whether his methods can be adapted to the circle case. 

\vspace{.3cm}

\item The full Lebesgue measure set of rotation numbers mentioned in the statement of Theorem \ref{rigC4}(3) -- call it $\mathbb{A}$ -- for which $C^{1+\alpha}$ rigidity holds, was originally defined in \cite[\S 4.4]{edsonwelington1} by means of three conditions on the sequence $(a_n)$ partial quotients, to wit:
\begin{itemize}
 \item $ \displaystyle{
\lim\sup_{n\to\infty}\frac{1}{n}\sum_{j=1}^{n}\log{a_j}\;<\; \infty
\ ;}
$
\item 
$
\displaystyle{
\lim_{n\to\infty}\frac{1}{n}\log{a_n}\;=\;0 \ ;}
$
\item 
$
\displaystyle{
\frac{1}{n}\sum_{j=k+1}^{k+n}\log{a_j}\;\leq\; \omega\left(\frac{n}{k}\right)
\ ,\ \ \ \forall\; 0<n\leq k\ .}
$
\end{itemize}
In this last condition, $\omega(t)$ is a positive function (that depends on
the rotation number) defined for $t>0$ such that $t\omega(t)\to 0$ as $t\to 0$. A natural question is: Are these conditions optimal? In other words, is $\mathbb{A}$ the largest set of rotation numbers for which statement (3) in Theorem \ref{rigC4} is true? 

\vspace{.3cm}

\item Another way to approach the question formulated in (3) is to look at the complement of the set $\mathbb{A}$. In \cite[\S 5]{edsonwelington1}, a saddle-node surgery technique was used to build $C^\infty$ counterexamples to $C^{1+\alpha}$ rigidity for each rotation number $\rho =[a_0,a_1,a_2,\ldots]$ satisfying $a_n\geq 2$ for all $n$ and 
\[
 \limsup_{n\to\infty} \frac{1}{n}\log{a_n} = \infty\ .
\]
Can such counterexamples be built for every rotation number not in $\mathbb{A}$? An analogous question can be asked in the analytic category. In \cite{avila}, using parabolic surgery, Avila constructed real-analytic counterexamples to $C^{1+\alpha}$ rigidity for each rotation number in another set of rotation numbers (still properly contained in the complement of $\mathbb{A}$). What is the optimal class of rotation numbers in this case? Is it still the whole complement of $\mathbb{A}$? 

\vspace{.3cm}

\item The problem of global hyperbolicity of the renormalization operator for $C^r$ \emph{unimodal maps} was solved in \cite{dFdMP}, through a combination of the deep holomorphic results obtained by Lyubich in \cite{lyubichAnnals} (later improved by Avila and Lyubich in \cite{AL}) with certain techniques of non-linear funcional analysis borrowed from the work of Davie \cite{davie}. Can these ideas be adapted to the study of the renormalization of $C^r$ (multi)critical circle maps? There are several difficulties to overcome here, such as to provide a suitable definition of a manifold structure in the space of real analytic critical commuting pairs (recall Section \ref{sechypren}). If this space can be endowed with a Banach manifold structure under which the renormalization operator is hyperbolic, then it is not too difficult to push such hyperbolicity to the space of $C^r$ critical commuting pairs (see \cite{voutaz}).

\vspace{.3cm}

\item Is it possible to prove Theorem \ref{geoconvC3BT} and Theorem \ref{geoconvC4} without appeal to holomorphic methods? Although quite powerful, the use of holomorphic methods limits the discussion to maps all of whose critical points have \emph{integral criticalities}.
See problem (2) above.

\vspace{.3cm}

\item Is Theorem \ref{rigC4} still true if the maps $f$ and $g$ are only $C^3$?

\vspace{.3cm}

\item Rigidity in the space of $C^2$ maps is most likely false. Can one construct explicit examples? What about rigidity in the space of $C^{2+\alpha}$ maps?

\vspace{.3cm}

\item\label{itemCvi} Finally, one topic that we did not touch at all in this survey is what is commonly referred to by physicists as \emph{mode locking universality}. In a typical (monotone) one-parameter family of (uni)critical circle maps, such as the Arnol'd family, the set of parameters for which the rotation number is irrational constitutes a Cantor set (see figure \ref{devilstair}), called the \emph{mode-locking Cantor set}. It is conjectured that the Hausdorff dimension of this Cantor set is a \emph{universal number} -- which has been numerically computed to be approximately $0.870...$, see \cite{CGV1990}. It has been shown by Grackzyk and \'Swi\k{a}tek \cite{graswia} that this dimension indeed lies strictly between zero and one. In particular, the Cantor set in question has zero Lebesgue measure. This is in sharp contrast with what happens in typical one-parameter families of circle \emph{diffeomorphisms}: in such cases, Herman \cite{H} had already shown in the seventies that the corresponding Cantor set has positive measure. Universality of the Hausdorff dimension in the critical case would follow from a careful study of the holonomy of the lamination determined by the stable manifolds of the renormalization operator (acting on a suitable space of critical commuting pairs), presumably in a similar manner as in the corresponding study of holonomy carried out for $C^r$ unimodal maps in \cite{dFdMP}. For more on the empirical study of the scaling geometry of the mode-locking Cantor set, see the work by Cvitanovi\'{c}, Shraiman and S\"{o}derberg~\cite{CvShSo85}.
\end{enumerate}

\end{document}